\documentclass[12pt, letterpaper]{article}

\usepackage{geometry}
\usepackage{color}
\usepackage{verbatim}
\usepackage[noadjust]{cite}
\usepackage[utf8x]{inputenc}
\usepackage{t1enc}
\usepackage[english]{babel}
\usepackage{pgfplots}
\pgfplotsset{compat=1.15}
\usepackage{mathrsfs}
\usetikzlibrary{arrows}

\definecolor{ffqqqq}{rgb}{1.,0.,0.}
\definecolor{uuuuuu}{rgb}{0.26666666666666666,0.26666666666666666,0.26666666666666666}

\usepackage{amsmath,amssymb,amsthm}
\usepackage{mathrsfs,esint}
\usepackage{graphicx,wrapfig}
\usepackage[format=hang]{caption}
\newcommand{\R}{\mathbb{R}}

\newcommand{\N}{\mathbb{N}}

\newcommand{\D}{\mathbb{D}}

\newcommand{\Cc}{\mathcal{C}}
\newcommand{\Ha}{\mathcal{H}}
\newcommand{\Leb}{\mathcal{L}}

\newcommand{\Ss}{\mathcal{S}}

\newcommand{\eps}{\varepsilon}

\newcommand{\loc}{\text{loc}}

\newcommand{\phii}{\varphi}

\newcommand{\bmat}{\begin{bmatrix}}
\newcommand{\emat}{\end{bmatrix}}

\newcommand{\wtil}{\widetilde}

\providecommand*{\vint}[1]{\mathchoice
          {\mathop{\vrule width 5pt height 3 pt depth -2.5pt
                  \kern -9pt \kern 1pt\intop}\nolimits_{\kern -5pt{#1}}}
          {\mathop{\vrule width 5pt height 3 pt depth -2.6pt
                  \kern -6pt \intop}\nolimits_{\kern -3pt{#1}}}
          {\mathop{\vrule width 5pt height 3 pt depth -2.6pt
                  \kern -6pt \intop}\nolimits_{\kern -3pt{#1}}}
          {\mathop{\vrule width 5pt height 3 pt depth -2.6pt
                  \kern -6pt \intop}\nolimits_{\kern -3pt{#1}}}}

\DeclareMathOperator{\Ext}{Ext}

\DeclareMathOperator{\dist}{dist}
\DeclareMathOperator{\diam}{diam}
\DeclareMathOperator{\rad}{rad}
\DeclareMathOperator{\supt}{supt}

\numberwithin{equation}{section}
\theoremstyle{plain}
\newtheorem{thm}[equation]{Theorem}

\newtheorem{prop}[equation]{Proposition}
\newtheorem{cor}[equation]{Corollary}
\newtheorem{lem}[equation]{Lemma}

\theoremstyle{definition}

\newtheorem{defn}[equation]{Definition}

\newtheorem{remark}[equation]{Remark}

\newtheorem{example}[equation]{Example}

\makeatletter
\def\blfootnote{\xdef\@thefnmark{}\@footnotetext}
\makeatother

\begin{document}

\title{Non-locality, non-linearity, and existence of solutions to the Dirichlet problem for least gradient functions in metric measure spaces}


\author{Josh Kline\thanks{The author was partially supported by the NSF grants \#DMS-1800161 and \#DMS-2054960.  The author would like to thank Nageswari Shanmugalingam for her kind encouragement and many fruitful discussions regarding this project, as well as Marie Snipes for her helpful insight. The author would also like to thank Piotr Rybka for pointing out the reference \cite{RS2}, Panu Lahti and Gareth Speight for helpful feedback and suggestions, and finally the anonymous referees whose comments and corrections helped to improve this paper.  In particular the author would like to thank the referee for pointing out Example~\ref{ex:NotConverge}.} }

\date{August 14, 2022}

\maketitle

\begin{abstract}
We study the Dirichlet problem for least gradient functions for domains in metric spaces equipped with a doubling measure and supporting a (1,1)-Poincar\'e inequality when the boundary of the domain satisfies a positive mean curvature condition.  In this setting, it was shown by Mal\'y, Lahti, Shanmugalingam, and Speight that solutions exist for continuous boundary data.  We extend these results, showing existence of solutions for boundary data that is approximable from above and below by continuous functions.  We also show that for each $f\in L^1(\partial\Omega),$ there is a least gradient function in $\Omega$ whose trace agrees with $f$ at points of continuity of $f$, and so we obtain existence of solutions for boundary data which is continuous almost everywhere.  This is in contrast to a result of Spradlin and Tamasan, who constructed an $L^1$-function on the unit circle which has no least gradient solution in the unit disk in $\R^2.$  Modifying the example of Spradlin and Tamasan, we show that the space of solvable $L^1$-functions on the unit circle is non-linear, even though the unit disk satisfies the positive mean curvature condition.

\end{abstract}

\section{Introduction} 
Given a function $f$ on the boundary of a domain $\Omega,$ the Dirichlet problem for least gradient functions is the problem of minimizing 
$\|Du\|(\Omega)$ over all $u\in BV(\Omega)$ with trace $Tu=f$ a.e.\ on the boundary.  This form of the problem, where the boundary condition is attained in the sense of traces, was originally introduced in the Euclidean setting by Sternberg, Williams, and Ziemer in \cite{SWZ}.  There they showed that if the boundary of the domain has non-negative mean curvature and is not locally area minimizing, then existence and uniqueness of solutions is guaranteed for continuous boundary data.  Furthermore, they showed that the imposed curvature conditions are necessary to guarantee existence of solutions, and if the boundary data is of class $C^{0,\alpha}$ for $0<\alpha\le 1$, then the solution is of class $C^{0,\alpha/2},$ provided the boundary of the domain has strictly positive mean curvature.  Their proof makes uses an important result of Bombieri, De Giorgi, and Giusti from \cite{BDG}, which states that characteristic functions of superlevel sets of least gradient functions are themselves of least gradient.  Using this, the authors constructed a least gradient solution by building its superlevel sets so that each was of least gradient and compatible with the boundary data. 

Since the appearance of \cite{SWZ}, existence, uniqueness, and regularity of the above least gradient problem have been studied extensively in the Euclidean setting.  For instance, we refer the interested reader to the following sampling \cite{G1,G2,Gnon,RS1,RS2,GRS,MRL} and the references therein.  In particular, it was shown in \cite{MRL} that there is an equivalence between least gradient solutions and solutions to the Dirichlet problem for the 1-Laplacian.  Moreover, the authors showed that in strictly convex domains, uniqueness of solutions may fail for discontinuous boundary data.  For more on the study of uniqueness of solutions, see \cite{Gnon}.  

In recent decades, a theory of analysis on metric measure spaces has been developed under the assumptions that the measure is doubling and that the space supports a Poincar\'e inequality, see for example \cite{BB, H ,HKST,AGS}.  Miranda~Jr.\ extended the definition of BV functions to this setting in \cite{MMJr}, leading to the development of a theory of least gradient functions and associated Dirichlet problems in metric spaces \cite{A,AMP,HKLS,KKLS,KLLS,LMSS,LS,MSS}.  In \cite{LMSS}, Lahti, Mal\'y, Shanmugalingam, and Speight studied the Dirichlet problem for least gradient functions, originally introduced in \cite{SWZ}, in the setting of a doubling metric measure space supporting a $(1,1)$-Poincar\'e inequality.  There they introduced a notion of positive mean curvature which makes sense in the metric setting (Definition~\ref{defn:MeanCurv} below).  They showed that if a domain satisfies this condition and if compatibility conditions are assumed between the measure and the codimension 1 Hausdorff measure of the boundary of the domain, then existence of solutions is guaranteed for continuous boundary data.  Their strategy adapts the argument from \cite{SWZ} to the metric setting, similarly building a solution by constructing its superlevel sets in an appropriate manner. 

 In contrast to \cite{SWZ}, \cite{LMSS} also provided examples in the weighted Euclidean setting which show that even for Lipschitz boundary data, solutions may fail to be continuous up to the boundary and may fail to be unique.  However, it was recently shown in \cite{Z} that continuous solutions exist for the weighted Euclidean least gradient problem with continuous boundary data, provided the weights are positive, bounded away from zero, and of class $C^2.$ This result, valid in dimensions $n\ge 2,$ extends the earlier result from \cite{JMN}, which guarantees existence of continuous solutions in low dimensions for $C^{1,1}$ weights which are positive and bounded away from zero.  In \cite{JMN}, it was also shown that for such weights and continuous boundary data, solutions to the weighted problem in dimensions $n\ge 2$ are unique.  For more on the weighted Euclidean least gradient problem, anisotropic formulations, and connections of these problems to conductivity imaging, see \cite{MNT,Mor,G1,G2}.  Such applications provide additional motivation for the study of the least gradient problem in the metric setting.  For a recent work on Gauss-Green formulas and connections to the least gradient problem in the metric setting, see \cite{GM}.

From \cite{SWZ} in the Euclidean setting, certain non-negative curvature conditions are required to guarantee existence of solutions to the Dirichlet problem for least gradients.  From \cite{AG,G}, it is also known that the the trace class of BV functions on a Euclidean Lipschitz domain is the $L^1$-class of its boundary.  In fact, analagous trace and extension results hold for BV functions in the metric setting as well, see \cite{LS,MSS}.  Therefore if a Euclidean domain satisfies the curvature conditions from \cite{SWZ}, it is natural to ask whether all $L^1$-functions on the boundary of such a domain admit solutions to the Dirichlet problem for least gradient functions.  This question was answered in the negative by Spradlin and Tamasan in \cite{ST}.  A certain fat Cantor set was constructed on the boundary of the unit disk in $\R^2$ in \cite{ST}, such that the characteristic function of that set is not the trace of a least gradient function in the unit disk, despite the fact that the unit disk satisfies the necessary curvature conditions.  Thus the question of which $L^1$-functions arise as traces of functions of least gradient is open even in the Euclidean setting.  The goal of this paper is to investigate the conditions sufficient to guarantee existence of solutions to the Dirichlet problem in both the Euclidean and metric settings.  

After introducing the necessary definitions and background information in Section~2, we begin Section 3 by examining the example presented in \cite{ST} in the Euclidean setting of the unit disk (Example~\ref{ex:GornyKa} below).  We modify this example in such a way as to obtain a solution (Example~\ref{ex:NonLin} below) which demonstrates that the set of $L^1$-functions on the boundary of the unit disk for which solutions exist is non-linear.  Namely, we show that the example function from \cite{ST} can be expressed as the sum of two functions, each of which arise as the trace of a least gradient function.  Moreover, our example shows that even in the case of the unit disk, the Dirichlet problem is non-local in the following sense.  There is a boundary data $f$ for which a least gradient solution to the Dirichlet problem exists, but $\eta f$ has no solution for a suitable compactly supported Lipschitz function $\eta$ on the boundary.  This example illustrates the significant difference between the Dirichlet problem for least gradient functions and the Dirichlet problem for $p$-harmonic functions when $p>1$, see \cite{BB}. 

In Sections 4 and 5, we obtain sufficient conditions for existence of solutions in the metric setting under the following standing assumptions: 
\begin{itemize}
\item $(X,d,\mu)$ is a complete metric measure space supporting a $(1,1)$-Poincar\'e inequality, with $\mu$ a doubling Borel regular measure.
\item $\Omega\subset X$ is a bounded domain with $\mu(X\setminus\Omega)>0.$
\item $\partial\Omega$ has positive mean curvature as in \cite{LMSS}, see Definition~\ref{defn:MeanCurv}.
\item $\Ha(\partial\Omega)<\infty$, $\Ha\big|_{\partial\Omega}$ is doubling, and $\Ha\big|_{\partial\Omega}$ is lower codimension 1 Ahlfors regular, see \eqref{eq:CodAhlfReg}.
\end{itemize}
Here, $\Ha$ is a codimension 1 Hausdorff measure on $\partial\Omega$, see \eqref{eq:Ha}.  The examples presented in Section 3 are in the setting of the unit disk in $\R^2,$ which satisfies the above assumptions as well.  The following is the first of the main results of the paper, proved in Section 4. 

\begin{thm}\label{thm:SeqApprox}  Let $f\in L^\infty(\partial\Omega)$, and for each $k\in\N$, let $g_k,h_k\in C(\partial\Omega)$ be such that $g_k,h_k\to f$ as $k\to\infty$ pointwise $\Ha$-a.e.\ on $\partial\Omega$, with 
$$g_k\le g_{k+1}\le f\le h_{k+1}\le h_k$$
$\Ha$-a.e.\ on $\partial\Omega.$  Then, there is a function $u\in BV(\Omega)$ which is the minimal solution to the Dirichlet problem with boundary data $f$.
\end{thm}

The following is an equivalent reformulation of the hypotheses of Theorem~\ref{thm:SeqApprox}:
\begin{itemize}
\item Let $f\in L^\infty(\partial\Omega)$ and assume that there is a lower semicontinuous function $g$ on $\partial\Omega$ and an upper semicontinuous function $h$ on $\partial\Omega$ such that $g\le h$ everywhere on $\partial\Omega$ and $g=h=f$ $\Ha$-a.e.\ on $\partial\Omega.$
\end{itemize}

The key step is to show existence of minimal solutions for continuous boundary data, from which we obtain a comparison theorem for minimal solutions; note that uniqueness of solutions is not guaranteed (see \cite{LMSS}), and so a more general comparison theorem will not hold for least gradient functions.  This is in contrast to $p$-harmonic functions with $p>1$, which always satisfy a comparison theorem (see \cite{BB}).  

As a consequence, we obtain the following result regarding characteristic functions of subsets of the boundary of the domain.  

\begin{thm}\label{thm:HZeroBoundary}
Let $F\subset\partial\Omega$ be measurable, and let $\wtil\partial F$ denote the boundary of $F$ relative to $\partial\Omega$. If $\Ha(\wtil\partial F)=0,$ then there is a function $u\in BV(\Omega)$ which is the minimal solution to the Dirichlet problem with boundary data $\chi_F.$ 
\end{thm}

 In Section 5, we continue to adopt the setting and assumptions from the previous section.  By adapting an argument from \cite{G2} in the Euclidean setting, we use the metric technology of discrete convolution to obtain the following result.

\begin{thm}\label{thm:CtyPointSoln}
Given $f\in L^1(\partial\Omega)$, there exists a least gradient function $u\in BV(\Omega)$ such that for all $x\in\partial\Omega$ such that $f$ is continuous at $x$, we have that $Tu(x)=f(x).$  In particular, if $f$ is continuous $\Ha$-a.e.\ on $\partial\Omega,$ then there is a solution to the Dirichlet problem with boundary data $f$. 
\end{thm} 

 This result was established for strictly convex, Euclidean domains in \cite{G2}.  Our extension to the metric setting includes Euclidean domains that are not strictly convex but satisfy the positive mean curvature condition.  For example, the capped cylinder described in Remark~\ref{rem:Cylinder} below is not strictly convex, but satisfies the positive mean curvature condition.

Although Theorem~\ref{thm:SeqApprox}, Theorem~\ref{thm:HZeroBoundary}, and Theorem~\ref{thm:CtyPointSoln} provide sufficient conditions on the boundary data to guarantee existence of solutions in this setting, Example~\ref{ex:NonLin} below shows that these conditions are not sharp (see Remarks~\ref{rem:NotSharp1} and \ref{rem:NotSharp2}).  It seems that even for sufficiently regular domains, a characterization of $L^1$ boundary data admitting solutions is still unknown.

\section{Preliminaries}

\subsection{General metric measure spaces and BV theory}

Throughout this paper, we assume that $(X,d,\mu)$ is a complete metric measure space equipped with a doubling Borel regular measure $\mu$.  By \emph{doubling}, we mean that there exists a constant $C_D\ge 1$ such that 
$$0<\mu(B(x,2r)\le C_D\mu(B(x,r))<\infty$$ for all $x\in X$ and $r>0.$  By iterating the doubling condition, there are constants $C\ge 1$ and $Q>1$ such that 
\begin{equation}\label{eq:LMBExp}
\frac{\mu(B(y,r))}{\mu(B(x,R))}\ge C^{-1}\left(\frac{r}{R}\right)^Q
\end{equation}
for every $0<r\le R$ and $y\in B(x,R).$

A complete metric space equipped with a doubling measure is proper, that is, closed and bounded sets are compact.  Thus for any open set $\Omega\subset X$, we define $L_{\loc}^1(\Omega)$ as the space of functions that are in $L^1(\Omega')$ for every $\Omega'\Subset\Omega,$ i.e., for every open set $\Omega'$ such that $\overline{\Omega'}$ is a compact subset of $\Omega.$  Also, if $A$ and $B$ are subsets of $X$, we use the notation $A\sqsubset B$ to mean that $\mu(A\setminus B)=0.$  By a {\it domain}, we mean a non-empty connected open set in $X$.  

Given a function $u:X\to\overline\R,$ we say that a Borel function $g:X\to[0,\infty]$ is an \emph{upper gradient} of $u$ if the following inequality holds for all non-constant compact rectifiable curves $\gamma:[a,b]\to X,$
$$|u(y)-u(x)|\le\int_{\gamma} g\,ds,$$
whenever $u(x)$ and $u(y)$ are both finite, and $\int_{\gamma} g\,ds=\infty$ otherwise.  Here $x$ and $y$ denote the endpoints of the curve $\gamma$.  Upper gradients were originally introduced in \cite{HeK}.  

Let $\widetilde N^{1,1}(X)$ be the class of all functions in $L^1(X)$ for which there exists an upper gradient in $L^1(X).$  For $u\in\widetilde N^{1,1}(X),$ we define 
$$\|u\|_{\widetilde N^{1,1}(X)}=\|u\|_{L^1(X)}+\inf_g\|g\|_{L^1(X)},$$
where the infimum is taken over all upper gradients $g$ of $u$.  Now, we define an equivalence relation in $\widetilde N^{1,1}(X)$ by $u\sim v$ if and only if $\|u-v\|_{\widetilde N^{1,1}(X)}=0.$  

The \emph{Newtonian space} $N^{1,1}(X)$ is defined as the quotient $\widetilde N^{1,1}(X)/\sim,$ and it is equipped with the norm $\|u\|_{N^{1,1}(X)}=\|u\|_{\widetilde N^{1,1}(X)}.$  One can analogously define $N^{1,1}(\Omega)$ for an open set $\Omega\subset X.$ For more on Newtonian spaces, see \cite{S}, \cite{HKST}, or \cite{BB}.

We now define functions of bounded variation on metric spaces, following the definition introduced by Miranda~Jr.\ in \cite{MMJr}.  For $u\in L^1_{\loc}(X),$ we define the \emph{total variation} of $u$  by 
\begin{equation*}
\|Du\|(X)=\inf\left\{\liminf_{i\to\infty}\int_X g_{u_i}\,d\mu:N^{1,1}_\loc(X)\ni u_i\to u\text{ in }L^1_{\loc}(X)\right\},
\end{equation*}
where $g_{u_i}$ are upper gradients of $u_i.$  For an open set $\Omega\subset X,$ we analogously define $\|Du\|(\Omega),$ and for an arbitrary $A\subset X,$ we define 
$$\|Du\|(A)=\inf\left\{\|Du\|(\Omega):A\subset\Omega, \Omega\subset X\text{ open}\right\}.$$  For $u\in L^1(X),$ we say that $u\in BV(X)$ ($u$ is of \emph{bounded variation}) if $\|Du\|(X)<\infty.$ We equip $BV(X)$ with the norm
\begin{equation*}
\|u\|_{BV(X)}=\|u\|_{L^1(X)}+\|Du\|(X).
\end{equation*}
We note that this definition coincides with the standard definition of the BV class in the Euclidean setting, see for example \cite{EG,AFP}.  See also \cite{A} and \cite{AMP} for more on BV theory in the metric setting.   

For $u\in BV(X),$ it was shown in \cite[Theorem~3.4]{MMJr} that $\|Du\|(\cdot)$ is a finite Radon measure on $X.$  Moreover, for an open set $\Omega\subset X,$ if $u_k\to u$ in $L^1_{\loc}(\Omega),$ then
\begin{equation}\label{eq:LowerSemiCty}
\|Du\|(\Omega)\le\liminf_{k\to\infty}\|Du_k\|(\Omega).
\end{equation}
That is, the BV energy is lower semi-continuous with respect to convergence in $L^1$. \cite[Proposition~3.6]{MMJr}

We say that a measurable set $E\subset X$ is of \emph{finite perimeter} if $\|D\chi_E\|(X)<\infty,$ and we denote the \emph{perimeter} of $E$ in $\Omega$ by 
$$P(E,\Omega):=\|D\chi_E\|(\Omega).$$

We have the following coarea formula, given by \cite[Proposition~4.2]{MMJr}. If $\Omega\subset X$ is an open set and $u\in L^1_\loc(\Omega),$ then 
\begin{equation}\label{eq:Coarea}
\|Du\|(\Omega)=\int_{-\infty}^\infty P(\{u>t\},\Omega)\,dt,
\end{equation}
and if $u\in BV(X),$ then the above holds with $\Omega$ replaced by any Borel set $A\subset\Omega.$ 

\subsection{Poincar\'e inequality and consequences}

We will also assume throughout this paper that $X$ supports a $(1,1)$-\emph{Poincar\'e inequality}, meaning that there are positive constants $\lambda$ and $C_P$ such that for every ball $B=B(x,r),$ every locally integrable function $u$, and every upper gradient $g$ of $u$, we have that 
\begin{equation*}
\fint_B|u-u_B|d\mu\le C_Pr\fint_{\lambda B}g\,d\mu,
\end{equation*}
where $\lambda B:=B(x,\lambda r),$ and
$$u_B:=\fint_B u\,d\mu=\frac{1}{\mu(B)}\int_B u\,d\mu.$$ 
Throughout this paper, we let $C$ denote a constant which depends, unless otherwise noted, on $C_D,$ $C_P,$ $\lambda$, or $\Omega.$  Its precise value is not of interest here and may not be the same at each occurrence.   

As shown in \cite{HK}, when $\mu$ is doubling, the (1,1)-Poincar\'e inequality implies the following Sobolev-Poincar\'e inequality,
$$\left(\fint_B|u-u_B|^{\frac{Q}{Q-1}}d\mu\right)^{\frac{Q-1}{Q}}\le C\rad(B)\fint_{\lambda B}g_u\,d\mu,$$ where $Q>1$ is the exponent from \eqref{eq:LMBExp}. Given $u\in L^1_\loc(X),$ one can apply this inequality to the approximating functions in $N^{1,1}(X)$ in the definition of total variation to obtain the inequality
$$\left(\fint_B|u-u_B|^{\frac{Q}{Q-1}}d\mu\right)^{\frac{Q-1}{Q}}\le C\rad(B)\frac{\|Du\|(2\lambda B)}{\mu(2\lambda B)},$$ from which the following lemma is obtained in \cite{KKLS}.

\begin{lem}\label{lem:MazyaIneq}\emph{(\cite[Lemma~2.2]{KKLS})}
Let $u\in BV(X)$, and for a ball $B\subset X$, let $$A=\{x\in B:|u(x)|>0\}.$$  If $\mu(A)\le\gamma\mu(B)$ for some $0<\gamma<1,$ then 
$$\left(\fint_B|u|^{\frac{Q}{Q-1}}d\mu\right)^{\frac{Q-1}{Q}}\le\frac{C\rad(B)}{1-\gamma^{1/Q}}\frac{\|Du\|(2\lambda B)}{\mu(2\lambda B)},$$ 
where $Q>1$ is the lower mass bound exponent given in \eqref{eq:LMBExp}.  
\end{lem}

We will use the above lemma in the proof of Lemma \ref{lem:DiscConvBound} to obtain $L^1$-bounds for a sequence of BV functions.

Given $E\subset X$, we define its \emph{codimension $1$ Hausdorff measure}, $\Ha(E)$, by 
\begin{equation}\label{eq:Ha}
\Ha(E)=\lim_{\delta\to 0^+}\inf\left\{\sum_i\frac{\mu(B_i)}{\rad(B_i)}: B_i\text{ balls in }X,\, E\subset\bigcup_i B_i,\,\rad(B_i)<\delta\right\}.
\end{equation}
We say that $\Ha\big|_{\partial\Omega}$ is \emph{lower codimension $1$ Ahlfors regular} if there exists $C>0$ such that 
\begin{equation}\label{eq:CodAhlfReg}
\Ha(B(x,r)\cap\partial\Omega)\ge C\frac{\mu(B(x,r))}{r}
\end{equation}
for every $x\in\partial\Omega$ and $0<r<2\diam(\partial\Omega).$

It was shown in \cite{A} and \cite{AMP} that if $\mu$ is doubling and $X$ supports a $(1,1)$-Poincar\'e inequality, then there is a constant $C\ge 1$ such that whenever $E\subset X$ is of finite perimeter and $A\subset X$ is a Borel set, we have
$$C^{-1}\Ha(A\cap\partial_M E)\le P(E,A)\le C\Ha(A\cap\partial_M E),$$ 
where $\partial_M E$ is the \emph{measure-theoretic boundary} of $E$, which is the set of all points $x\in X$ for which 
$$\limsup_{r\to 0^+}\frac{\mu(B(x,r)\cap E)}{\mu(B(x,r))}>0\quad\text{and}\quad\limsup_{r\to 0^+}\frac{\mu(B(x,r)\setminus E)}{\mu(B(x,r))}>0.$$  

Given an extended real-valued function $u$ on $X$, we define the approximate upper and lower limits of $u$ by 
\begin{align*}
u^\vee(x)&:=\inf\left\{t\in\R:\lim_{r\to 0^+}\frac{\mu(\{u>t\}\cap B(x,r))}{\mu(B(x,r))}=0\right\},\\
u^\wedge(x)&:=\sup\left\{t\in\R:\lim_{r\to 0^+}\frac{\mu(\{u<t\}\cap B(x,r))}{\mu(B(x,r))}=0\right\}.
\end{align*}
From the Lebesgue differentiation theorem, $u^\vee=u^\wedge$ $\mu$-a.e.\ if $u\in L^1_\loc(X).$ 

\subsection{Dirichlet problem for least gradient functions}

\begin{defn}\label{defn:Trace} Given a bounded domain $\Omega\subset X$ and a function $u\in BV(\Omega),$ we say that $u$ has a \emph{trace} at a point $x\in\partial\Omega$ if there is a number $Tu(x)\in\R$ such that 
$$\lim_{r\to 0^+}\fint_{B(x,r)\cap\Omega}|u-Tu(x)|d\mu=0.$$
\end{defn}

\begin{defn}\label{defn:LeastGradient}  Let $\Omega\subset X$ be an open set, and let $u\in BV_{\loc}(\Omega).$  We say that $u$ is of \emph{least gradient in }$\Omega$ if
$$\|Du\|(V)\le\|Dv\|(V),$$
whenever $v\in BV(\Omega)$ with $\overline{\{x\in\Omega: u(x)\ne v(x)\}}\subset V\Subset\Omega.$
\end{defn}

\begin{defn}\label{defn:WeakSolution}  Let $\Omega$ be a bounded domain in $X$ with $\mu(X\setminus\Omega)>0,$ and let $f\in BV_{\loc}(X).$  We say that $u\in BV_{\loc}(X)$ is a \emph{weak solution to the Dirichlet problem for least gradients in $\Omega$ with boundary data $f$}, or simply, \emph{weak solution to the Dirichlet problem with boundary data $f$}, if $u=f$ on $X\setminus\Omega$ and 
$$\|Du\|(\overline\Omega)\le\|Dv\|(\overline\Omega),$$
whenever $v\in BV(X)$ with $v=f$ on $X\setminus\Omega.$ 
\end{defn}

\begin{defn}\label{defn:Solution}  Let $\Omega$ be a domain in $X$ and $f:\partial\Omega\to\R.$  We say that a function $u\in BV(\Omega)$ is a \emph{solution to the Dirichlet problem for least gradients in $\Omega$ with boundary data $f$}, or simply, \emph{solution to the Dirichlet problem with boundary data $f$}, if $Tu=f$ $\Ha$-a.e.\ on $\partial\Omega$ and whenever $v\in BV(\Omega),$ with $Tv=f$ $\Ha$-a.e.\ on $\partial\Omega,$ we must have 
$$\|Du\|(\Omega)\le\|Dv\|(\Omega).$$ 
\end{defn}

Note that solutions and weak solutions to Dirichlet problems on a domain $\Omega$ are necessarily of least gradient in $\Omega.$ 

\begin{defn}\label{defn:MinimalSolution}
A (weak) solution $\chi_E$ to the Dirichlet problem with boundary data $\chi_F$ is called a \emph{minimal (weak) solution} to the said problem if every (weak) solution $\chi_{\widetilde E}$ corresponding to the data $\chi_F$ satisfies $E\sqsubset \widetilde E,$ that is, $\mu(E\setminus\widetilde E)=0$, or alternatively, $\chi_E\le\chi_{\widetilde E}$ $\mu$-a.e.\ in $X.$ 
\end{defn}

It is shown in \cite{LMSS} that if $F\subset X$ is such that $P(F,X)<\infty,$ then there is a set $E\subset X$ with $P(E,X)<\infty$ such that $\chi_E$ is a weak solution to the Dirichlet problem with boundary data $\chi_F.$  We call $E$ a \emph{weak solution set}.  Moreover, for such an $F$, there is a minimal weak solution, and such a minimal weak solution is unique $\mu$-a.e. in $X,$ \cite[Proposition~3.7]{LMSS}.  However, without additional assumptions on $\Omega,$ the trace of the weak solution may not agree with $\chi_F$ on $\partial\Omega.$  That is, a weak solution may not necessarily be a solution.  For example, if $\Omega=(0,1)\times(0,1)\subset\R^2,$ and $F$ is the disk centered at $(1/2,0)$ of radius $1/10,$ then the trace of the minimal weak solution will have zero trace on $\partial\Omega,$ and in fact there is no least gradient function with the appropriate trace on the boundary. To address this issue, the following definition was introduced in \cite{LMSS}, extending the formulation from \cite{SWZ} to the metric setting.

\begin{defn}\label{defn:MeanCurv}  Given a domain $\Omega\subset X$, we say that the boundary $\partial\Omega$ has \emph{positive mean curvature} if for each $x\in\partial\Omega,$ there exists a non-decreasing function $\phi_x:(0,\infty)\to(0,\infty)$ and a constant $r_x>0$ such that for all $0<r<r_x$ with $P(B(x,r),X)<\infty$, we have that $B(x,\phi_x(r))\sqsubset E_{B(x,r)},$ where $E_{B(x,r)}\subset X$ gives the minimal weak solution to the Dirichlet problem with boundary data $\chi_{B(x,r)},$ as defined above.
\end{defn}

In \cite{LMSS}, positive mean curvature is defined by existence of $\phi$ and $r_0>0$ so that the condition is satisfied for all $x\in\partial\Omega$ and all $0<r<r_0.$  However, the results from \cite{LMSS} hold if the definition is weakened to allow dependence on $x$, as above.

\begin{remark}\label{WeakSolnVSoln} It is shown in \cite[Proposition~4.8,~4.9]{LMSS} that if $\Ha(\partial\Omega)<\infty,$ and $F\subset X$ is open with $P(F,X)<\infty$ and $\Ha(\partial F\cap\partial\Omega)=0,$ then under the assumption of positive mean curvature, all weak solutions of $\chi_F$ are solutions.  Additionally, if $v\in BV(\Omega)$ is a solution for $\chi_F,$ then extending $v$ outside $\Omega$ by $\chi_F$ yields a weak solution.
\end{remark}

\section{Motivating examples}

The domain $\Omega= B(0,1)\subset\R^2$ has boundary of positive mean curvature as defined above, but it was shown by Spradlin and Tamasan in \cite{ST} that there exists a function $f\in L^1(\partial\Omega)$ for which there is no solution to the Dirichlet problem when $f$ is the given boundary data.  The function $f$ is the characteristic function of a certain fat Cantor set on the unit circle.  The following example, due to G\'orny \cite[Example~4.7]{G1}, is a modification of the example from \cite{ST}.

\begin{example}\label{ex:GornyKa}
Let $\Omega=B(0,1)\subset\R^2$.  We construct a Cantor set $K_{1/4}$ on the unit circle as follows.  Let $$I_0:=\left\{(\cos\theta,\sin\theta):\pi/2-1/2\le\theta\le\pi/2+1/2\right\},$$ and define $f_0:\partial\Omega\to\R$ by $f_0=\chi_{I_0}.$  To construct $f_1:\partial\Omega\to\R,$ we remove an arc of arc-length $1/4$ from the center of $I_0,$ and let 
$$I_{1,1}=\{(\cos\theta,\sin\theta):\pi/2-1/2\le\theta\le\pi/2-1/8\},$$
$$I_{1,2}=\{(\cos\theta,\sin\theta):\pi/2+1/8\le\theta\le\pi/2+1/2\},$$
and let $J_1:=I_{1,1}\cup I_{1,2}.$  Define $f_1:\partial\Omega\to\R$ by $f_1=\chi_{J_1}.$  

Continuing inductively in this manner, we construct $f_n$ from $f_{n-1}$ by removing an arc of arc-length $1/4^n$ from the center of $I_{n-1,m}$ for each $m\in\{1,\dots,2^{n-1}\},$ that is, from each arc comprising $J_{n-1}.$  We then obtain a new collection of arcs $\{I_{n,m}\}_{m=1}^{2^n},$ and by a direct computation, it follows that the arc length of each $I_{n,m}$ is given by 
$$\Ha(I_{n,m})=\frac{2^n+1}{2^{2n+1}}.$$
Setting $J_n=\bigcup_{m=1}^{2^n}I_{n,m},$ we define $f_n:\partial\Omega\to\R$ by $f_n=\chi_{J_n}$. The Cantor set is then given by $K_{1/4}=\bigcap_{n\in\N} J_n,$ and we define $f:\partial\Omega\to\R$ by $f=\chi_{K_{1/4}}.$  We note that $f\in L^1(\partial\Omega).$  

\begin{figure}[h]
\centering
\begin{minipage}[t]{0.45\textwidth}
\centering
\begin{tikzpicture}[line cap=round,line join=round,>=triangle 45,x=3.0cm,y=3.0cm]
\clip(-1.2,-1.2) rectangle (1.2,1.2);
\draw [line width=0.4pt] (0.,0.) circle (3.0046690567933103cm);
\draw [shift={(0.,0.)},line width=0.4pt,color=ffqqqq,fill=ffqqqq,fill opacity=0.3100000023841858]  plot[domain=1.9198621771937627:2.617993877991494,variable=\t]({1.*1.0015563522644368*cos(\t r)+0.*1.0015563522644368*sin(\t r)},{0.*1.0015563522644368*cos(\t r)+1.*1.0015563522644368*sin(\t r)});
\draw [shift={(0.,0.)},line width=0.4pt,color=ffqqqq,fill=ffqqqq,fill opacity=0.33000001311302185]  plot[domain=0.523598775598299:1.2217304763960306,variable=\t]({1.*1.001556352264437*cos(\t r)+0.*1.001556352264437*sin(\t r)},{0.*1.001556352264437*cos(\t r)+1.*1.001556352264437*sin(\t r)});
\draw [line width=1.2pt] (-0.34255244715021665,0.9411551135241432)-- (0.34255244715021665,0.9411551135241432);
\draw [line width=1.2pt] (-0.8673732443826784,0.5007781761322185)-- (0.8673732443826784,0.5007781761322185);
\draw [line width=1.2pt] (-0.8673732443826784,0.5007781761322185)-- (-0.34255244715021665,0.9411551135241432);
\draw [line width=1.2pt] (0.8673732443826784,0.5007781761322185)-- (0.34255244715021665,0.9411551135241432);
\begin{scriptsize}
\draw [fill=uuuuuu] (-0.8673732443826784,0.5007781761322185) circle (1.0pt);
\draw [fill=uuuuuu] (0.8673732443826784,0.5007781761322185) circle (1.0pt);
\draw [fill=uuuuuu] (-0.34255244715021665,0.9411551135241432) circle (1.0pt);
\draw [fill=uuuuuu] (0.34255244715021665,0.9411551135241432) circle (1.0pt);
\draw[color=ffqqqq] (-0.7448868047882511,0.9217262632502173) node {$I_{1,1}$};
\draw[color=ffqqqq] (0.7480655166932737,0.893557351524151) node {$I_{1,2}$};
\end{scriptsize}
\end{tikzpicture}
\caption{\small{$u_1=1$ in the shaded regions, and $u_1=0$ elsewhere in the disk.}  }\label{fig:f_1}
\end{minipage}
\hfill
\begin{minipage}[t]{0.45\textwidth}
\centering
\definecolor{qqqqff}{rgb}{0.,0.,1.}
\definecolor{ffqqqq}{rgb}{1.,0.,0.}
\begin{tikzpicture}[line cap=round,line join=round,>=triangle 45,x=3.5cm,y=3.5cm]
\clip(-1.,-0.5) rectangle (1.,1.2);
\draw [line width=0.8pt,color=ffqqqq] (-0.7676043049047092,0.3013032211772025)-- (0.7678409576922745,0.3006996236948274);
\draw [line width=0.4pt] (-0.5412470074860535,0.6221347738934001)-- (-0.3062110369032702,0.7656597161132511);
\draw [line width=0.4pt] (0.10265086657297898,0.8182070640075265)-- (0.5944352024707327,0.5715302179793987);
\draw [shift={(0.,0.)},line width=1.2pt,color=qqqqff,fill=qqqqff,fill opacity=0.36000001430511475]  plot[domain=0.7657560382254658:1.4459901093726175,variable=\t]({1.*0.8246211251235321*cos(\t r)+0.*0.8246211251235321*sin(\t r)},{0.*0.8246211251235321*cos(\t r)+1.*0.8246211251235321*sin(\t r)});
\draw [shift={(0.,0.)},line width=1.2pt,color=qqqqff,fill=qqqqff,fill opacity=0.36000001430511475]  plot[domain=1.951243198315623:2.286778067362325,variable=\t]({1.*0.8246211251235321*cos(\t r)+0.*0.8246211251235321*sin(\t r)},{0.*0.8246211251235321*cos(\t r)+1.*0.8246211251235321*sin(\t r)});
\draw [shift={(0.,0.)},line width=0.4pt]  plot[domain=0.24497866312686414:0.7657560382254658,variable=\t]({1.*0.8246211251235321*cos(\t r)+0.*0.8246211251235321*sin(\t r)},{0.*0.8246211251235321*cos(\t r)+1.*0.8246211251235321*sin(\t r)});
\draw [shift={(0.,0.)},line width=0.4pt]  plot[domain=1.4459901093726175:1.951243198315623,variable=\t]({1.*0.8246211251235323*cos(\t r)+0.*0.8246211251235323*sin(\t r)},{0.*0.8246211251235323*cos(\t r)+1.*0.8246211251235323*sin(\t r)});
\draw [shift={(0.,0.)},line width=0.4pt]  plot[domain=2.286778067362325:2.896613990462929,variable=\t]({1.*0.8246211251235321*cos(\t r)+0.*0.8246211251235321*sin(\t r)},{0.*0.8246211251235321*cos(\t r)+1.*0.8246211251235321*sin(\t r)});
\draw [shift={(0.,0.)},line width=0.4pt,dash pattern=on 1pt off 1pt,color=ffqqqq]  plot[domain=0.37325888709703364:2.7675475483630922,variable=\t]({1.*1.0023190006591483*cos(\t r)+0.*1.0023190006591483*sin(\t r)},{0.*1.0023190006591483*cos(\t r)+1.*1.0023190006591483*sin(\t r)});
\draw [line width=0.4pt,dash pattern=on 1pt off 1pt,color=ffqqqq] (0.7678409576922745,0.3006996236948274)-- (0.933303256406378,0.3654974837992789);
\draw [line width=0.4pt,dash pattern=on 1pt off 1pt,color=ffqqqq] (-0.9359384140660347,0.36737842294486134)-- (-0.7676043049047092,0.3013032211772025);
\begin{scriptsize}
\draw [fill=black] (-0.7676043049047092,0.3013032211772025) circle (1.0pt);
\draw [fill=black] (0.7678409576922745,0.3006996236948274) circle (1.0pt);
\draw [fill=black] (0.5944352024707327,0.5715302179793987) circle (1.0pt);
\draw [fill=black] (0.10265086657297898,0.8182070640075265) circle (1.0pt);
\draw [fill=black] (-0.3062110369032702,0.7656597161132511) circle (1.0pt);
\draw [fill=black] (-0.5412470074860535,0.6221347738934001) circle (1.0pt);
\draw[color=ffqqqq] (0.030553183265790172,0.508856250965606) node {$\Omega_A$};
\draw[color=ffqqqq] (-0.015242537413805698,1.1125271144693694) node {$A$};
\end{scriptsize}
\end{tikzpicture}
\caption{$w=0$ in the shaded regions, and $w=1$ elsewhere in $\Omega_A$.}\label{fig:Omega_A}
\end{minipage}
\end{figure}

For each $n\in\N,$ consider the function $u_n:\Omega\to\R$ given as follows.  For each $m\in\{1,\dots, 2^n\},$ let $u_n=1$ on the region of $\Omega$ bounded by $I_{n,m}$ and the chord joining the endpoints of $I_{n,m}.$  Let $u_n=0$ elsewhere in $\Omega$.  It was shown in \cite[Example~4.7]{G1} that $u_n$ is a solution to the Dirichlet problem with boundary data $f_n.$  This was done by considering the trapezoid formed by the chord joining the endpoints of each arc $I_{n-1,m}$, the chord joining the endpoints of the arc removed from the center of $I_{n-1,m},$ and the chords joining the endpoints of the two arcs remaining after the removal. Since the Cantor set was constructed using the removal parameter 1/4, it was shown that the sum of the lengths of the bases of the trapezoid is greater than the sum of the lengths of the sides, and such an inequality holds on every stage of the construction (see Figure \ref{fig:f_1}).  Here by bases, we mean the two parallel sides of the trapezoid.

\vskip .2cm

\noindent{\it Claim.}  If $n\in\N$ is sufficiently large, then for any set $E\subset\Omega$ such that $w=\chi_E$ is a solution to the Dirichlet problem with boundary data $f_n$, we have that $w=u_n$ a.e.\ in $\Omega.$ That is, for sufficiently large $n\in\N$, solution sets to the Dirichlet problem with boundary data $f_n$ are unique a.e.\  

\vskip .2cm

\noindent{\it Proof of claim.}  By summing the lengths of the line segments which comprise the boundary of $\{u_n=1\},$ we have by direct computation that $$\|Du_n\|(\Omega)\to\Ha(K_{1/4})=1/2$$ as $n\to\infty.$  Since $2\sin(5/16)>1/2,$ there exists $N\in\N$ such that for all $n\in\N$, $n>N$ implies that 
\begin{equation}\label{eq:u_nbound}
\|Du_n\|(\Omega)<2\sin(5/16).
\end{equation}
Fix $n>N$ and suppose that there exists $E\subset\Omega$ such that $w=\chi_E$ is a solution to the Dirichlet problem with boundary data $f_n.$  Since $w$ is of least gradient in $\Omega,$ we have that $\partial E\cap\Omega$ consists of straight line segments.  We will show that multiple arcs from $\{I_{n,m}\}_{m=1}^{2^n}$ cannot be contained in the boundary of a single connected component of $E.$  In doing so, this will show that $w=u_n$ a.e.\ in $\Omega.$  

Suppose that a connected component $E_0$ of $E$ contains multiple arcs from $\{I_{n,m}\}_{m=1}^{2^n}$ it its boundary.  Let $I_{n,m_1}$ and $I_{n,m_2}$ be the two arcs forming part of the boundary of $E_0$ farthest from one another on $\partial\Omega$, and let $A$ be the shortest arc on $\partial\Omega$ which contains both $I_{n,m_1}$ and $I_{n,m_2}.$  Since the chord joining the endpoints of $A$ forms part of the perimeter of $E,$ we have that 
$$\|Dw\|(\Omega)\ge 2\sin(\Ha(A)/2).$$
If $\Ha(A)\ge 5/8,$ then by the choice of $N$ and \eqref{eq:u_nbound}, we would have that 
$$\|Dw\|(\Omega)\ge 2\sin(5/16)>\|Du_n\|(\Omega).$$  However, this contradicts $w$ being a solution to the Dirichlet problem with boundary data $f_n$, and so it follows that $\Ha(A)<5/8.$  

Let $\Omega_A$ denote the region of $\Omega$ bounded by $A$ and the chord joining the endpoints of $A$.  Since $\Ha(A)<5/8,$ it follows from the argument in Example~\ref{ex:NonLin} that for each subarc of $A$ which was removed in the construction of $f_n$, $w=0$ on the region of $\Omega_A$ bounded by that subarc and the chord joining its endpoints.  Likewise, $w=1$ elsewhere in $\Omega_A.$  It is shown, as part of the discussion of Example~\ref{ex:NonLin}, that this configuration minimizes the perimeter of potential solution sets in such regions $\Omega_A,$ see Figure~\ref{fig:Omega_A}.

Let $C$ denote the largest subarc of $A$ removed during the construction of $f_n,$ and let $k\in\N$ such that $\Ha(C)=4^{-k}$; we note that $k\le n.$  Consider the function $h:[0,\Ha(I_{k,m})]\to\R,$ given by 
$$h(\theta)=2\left[\sin\left(\frac{\Ha(C)}{2}\right)+\sin\left(\frac{\Ha(I_{k,m})+\Ha(C)+\theta}{2}\right)-\sin\left(\frac{\Ha(I_{k,m})}{2}\right)-\sin\left(\frac{\theta}{2}\right)\right],$$ 
where $I_{k,m}$ is one of the two arcs adjacent to $C$ at the $k$-th stage of the construction.  The function $h$ measures the difference between the sum of the lengths of the bases and the sum of the lengths of the sides of the quadrilateral shown in Figure~\ref{fig:hfunction}.  Here, by bases we mean the chord joining the end points of $C$ and the side of the quadrilateral opposite that chord.  Because the trapezoid inequality between side and base lengths discussed above was shown to hold at every stage of the construction in \cite[Example~4.7]{G1}, we have that $h\left(\Ha(I_{k,m})\right)>0.$  Furthermore, we have that 
$$h'(\theta)=\cos\left(\frac{\Ha(I_{k,m})+\Ha(C)+\theta}{2}\right)-\cos\left(\frac{\theta}{2}\right)<0,$$ and so $h$ is positive on $[0,\Ha(I_{k,m})].$  Similarly, for any fixed $\lambda\in[0,\Ha(I_{k,m})],$ the function $h_\lambda:[0,\Ha(I_{k,m})]\to\R$ given by 
$$h_\lambda(\theta)=2\left[\sin\left(\frac{\Ha(C)}{2}\right)+\sin\left(\frac{\lambda+\Ha(C)+\theta}{2}\right)-\sin\left(\frac{\lambda}{2}\right)-\sin\left(\frac{\theta}{2}\right)\right]$$ is decreasing with $h_\lambda(\Ha(I_{k,m}))>0.$  Hence, $h_\lambda$ is positive on $[0,\Ha(I_{k,m})].$  

Thus the sum of the length of the the chord joining the endpoints of $C$ and the length of the chord joining the endpoints of $A$ is strictly greater than the sum of the lengths of the chords which join each endpoint of $C$ to the corresponding endpoint of $A$, see Figure~\ref{fig:sidelengths}.  This is a contradiction, since $w$ is a solution, and the chord joining the endpoints of $A$ and the chord joining the endpoints of $C$ form part of the perimeter of $A$ in $\Omega.$  Hence, multiple arcs from $\{I_{k,m}\}_{m=1}^{2^n}$ cannot be contained in the boundary of a single connected component of $E$.  Therefore, we have that $w=u_n$ a.e.\ in $\Omega,$ proving the claim.

\begin{figure}[h]
\centering
\begin{minipage}[t]{0.45\textwidth}
\centering
\definecolor{ffqqqq}{rgb}{1.,0.,0.}
\definecolor{qqqqff}{rgb}{0.,0.,1.}
\begin{tikzpicture}[line cap=round,line join=round,>=triangle 45,x=3.5cm,y=3.5cm]
\clip(-1.,0.) rectangle (1.,1.2);
\draw [shift={(0.,0.)},line width=1.2pt,color=qqqqff]  plot[domain=1.1608310274339484:1.9872279318512736,variable=\t]({1.*0.8246211251235321*cos(\t r)+0.*0.8246211251235321*sin(\t r)},{0.*0.8246211251235321*cos(\t r)+1.*0.8246211251235321*sin(\t r)});
\draw [line width=0.8pt] (-0.7676043049047092,0.3013032211772025)-- (-0.3335589168491971,0.7541474981662344);
\draw [line width=0.8pt] (0.32867542902704827,0.7562886104880105)-- (0.6649028198962915,0.4877542824967915);
\draw [shift={(0.,0.)},line width=0.4pt,color=ffqqqq]  plot[domain=1.9872279318512736:2.7675475483630922,variable=\t]({1.*0.8246211251235321*cos(\t r)+0.*0.8246211251235321*sin(\t r)},{0.*0.8246211251235321*cos(\t r)+1.*0.8246211251235321*sin(\t r)});
\draw [shift={(0.,0.)},line width=0.4pt,color=ffqqqq]  plot[domain=0.6329042065078376:1.1608310274339484,variable=\t]({1.*0.8246211251235323*cos(\t r)+0.*0.8246211251235323*sin(\t r)},{0.*0.8246211251235323*cos(\t r)+1.*0.8246211251235323*sin(\t r)});
\draw [line width=0.8pt] (-0.3335589168491971,0.7541474981662344)-- (0.32867542902704827,0.7562886104880105);
\draw [line width=0.8pt] (-0.7676043049047092,0.3013032211772025)-- (0.6649028198962915,0.4877542824967915);
\draw [shift={(0.,0.)},line width=0.4pt]  plot[domain=0.24497866312686414:0.6329042065078376,variable=\t]({1.*0.8246211251235321*cos(\t r)+0.*0.8246211251235321*sin(\t r)},{0.*0.8246211251235321*cos(\t r)+1.*0.8246211251235321*sin(\t r)});
\draw [shift={(0.,0.)},line width=0.4pt]  plot[domain=2.7675475483630922:2.896613990462929,variable=\t]({1.*0.824621125123532*cos(\t r)+0.*0.824621125123532*sin(\t r)},{0.*0.824621125123532*cos(\t r)+1.*0.824621125123532*sin(\t r)});
\begin{scriptsize}
\draw [fill=black] (-0.7676043049047092,0.3013032211772025) circle (1.0pt);
\draw [fill=black] (0.6649028198962915,0.4877542824967915) circle (1.0pt);
\draw [fill=black] (0.32867542902704827,0.7562886104880105) circle (1.0pt);
\draw [fill=black] (-0.3335589168491971,0.7541474981662344) circle (1.0pt);
\draw[color=qqqqff] (-0.026294911285031475,0.8951382817472405) node {$C$};
\draw[color=ffqqqq] (-0.7150150726739684,0.6440315360668564) node {$I_{k,m}$};
\draw[color=ffqqqq] (0.62387359205339,0.6867509409336436) node {$\theta$};
\end{scriptsize}
\end{tikzpicture}
\caption{}\label{fig:hfunction}
\end{minipage}
\hfill
\begin{minipage}[t]{0.45\textwidth}
\centering
\definecolor{ffqqqq}{rgb}{1.,0.,0.}
\definecolor{qqqqff}{rgb}{0.,0.,1.}
\begin{tikzpicture}[line cap=round,line join=round,>=triangle 45,x=3.5cm,y=3.5cm]
\clip(-1.,0.) rectangle (1.,1.4);
\draw [line width=0.4pt] (-0.5412470074860535,0.6221347738934001)-- (-0.3062110369032702,0.7656597161132511);
\draw [shift={(0.,0.)},line width=1.2pt,color=qqqqff]  plot[domain=0.9034548843869886:1.5072644527326065,variable=\t]({1.*0.8246211251235321*cos(\t r)+0.*0.8246211251235321*sin(\t r)},{0.*0.8246211251235321*cos(\t r)+1.*0.8246211251235321*sin(\t r)});
\draw [shift={(0.,0.)},line width=0.8pt,color=qqqqff,fill=qqqqff,fill opacity=0.36000001430511475]  plot[domain=1.951243198315623:2.286778067362325,variable=\t]({1.*0.8246211251235321*cos(\t r)+0.*0.8246211251235321*sin(\t r)},{0.*0.8246211251235321*cos(\t r)+1.*0.8246211251235321*sin(\t r)});
\draw [line width=0.8pt,color=ffqqqq] (-0.7676043049047092,0.3013032211772025)-- (0.05235448914947508,0.8229574760981867);
\draw [line width=0.8pt,color=ffqqqq] (0.5103579812795977,0.6477150075026932)-- (0.7678409576922745,0.3006996236948274);
\draw [shift={(0.,0.)},line width=0.4pt]  plot[domain=0.24497866312686414:0.9034548843869886,variable=\t]({1.*0.8246211251235321*cos(\t r)+0.*0.8246211251235321*sin(\t r)},{0.*0.8246211251235321*cos(\t r)+1.*0.8246211251235321*sin(\t r)});
\draw [shift={(0.,0.)},line width=0.4pt]  plot[domain=1.5072644527326065:1.951243198315623,variable=\t]({1.*0.8246211251235321*cos(\t r)+0.*0.8246211251235321*sin(\t r)},{0.*0.8246211251235321*cos(\t r)+1.*0.8246211251235321*sin(\t r)});
\draw [shift={(0.,0.)},line width=0.4pt]  plot[domain=2.286778067362325:2.896613990462929,variable=\t]({1.*0.8246211251235321*cos(\t r)+0.*0.8246211251235321*sin(\t r)},{0.*0.8246211251235321*cos(\t r)+1.*0.8246211251235321*sin(\t r)});
\draw [line width=0.4pt,dash pattern=on 1pt off 1pt,color=ffqqqq] (0.7678409576922745,0.3006996236948274)-- (0.9754758754552496,0.38201300117456227);
\draw [shift={(0.,0.)},line width=0.4pt,dash pattern=on 1pt off 1pt,color=ffqqqq]  plot[domain=0.3732588870970336:2.7675475483630922,variable=\t]({1.*1.0476101930878592*cos(\t r)+0.*1.0476101930878592*sin(\t r)},{0.*1.0476101930878592*cos(\t r)+1.*1.0476101930878592*sin(\t r)});
\draw [line width=0.4pt,dash pattern=on 1pt off 1pt,color=ffqqqq] (-0.7676043049047092,0.3013032211772025)-- (-0.9770016837316641,0.38349674763802233);
\begin{scriptsize}
\draw [fill=black] (-0.7676043049047092,0.3013032211772025) circle (1.0pt);
\draw [fill=black] (0.7678409576922745,0.3006996236948274) circle (1.0pt);
\draw [fill=black] (0.5103579812795977,0.6477150075026932) circle (1.0pt);
\draw [fill=black] (0.05235448914947508,0.8229574760981867) circle (1.0pt);
\draw [fill=black] (-0.3062110369032702,0.7656597161132511) circle (1.0pt);
\draw [fill=black] (-0.5412470074860535,0.6221347738934001) circle (1.0pt);
\draw[color=qqqqff] (0.34966357721792113,0.8382604704437056) node {$C$};
\draw[color=ffqqqq] (1.77580785429704E-4,1.1852501383302496) node {$A$};
\end{scriptsize}
\end{tikzpicture}
\caption{}\label{fig:sidelengths}
\end{minipage}
\end{figure}

Now suppose that there exists a solution $u\in BV(\Omega)$ to the Dirichlet problem with boundary data $f=\chi_{K_{1/4}}$.  By Remark~\ref{remark:SupLevSoln} below, $\chi_{\{u>t\}}$ is a solution to the Dirichlet problem with boundary data $\chi_{\{f>t\}}$ for $\Leb$-a.e.\ $t\in\R.$  Thus we may assume that there exists some $E\subset\Omega$ such that $u=\chi_E.$  Since for all $n\in\N$, $f\le f_n$ on $\partial\Omega,$ it follows from Lemma~\ref{lem:MinMaxSoln} below that $\max\{u,u_n\}$ is a solution to the Dirichlet problem with boundary data $f_n.$  But $\max\{u,u_n\}$ is the characteristic function of a subset of $\Omega,$ and so by the claim above, $\max\{u,u_n\}=u_n$ a.e.\ in $\Omega$ for sufficiently large $n$.  Hence for sufficiently large $n$, $u\le u_n$ a.e. in $\Omega$, and since $u_n\to 0$ as $n\to\infty,$ it follows that $u=0$ a.e.\ in $\Omega.$  However $\Ha(K_{1/4})=1/2$, and for each $x\in K_{1/4},$ $Tu(x)=0\ne f(x)$, a contradiction.  Therefore there is no solution to the Dirichlet problem with boundary data $f$. 
\end{example}

In the next example, we show that a slight modification of the function $f$ constructed above, namely the addition of another arc to $K_{1/4},$ renders the new function solvable.  

\begin{example}\label{ex:NonLin}
Let $F:=\partial\Omega\setminus I_0.$ For each $n\in\N,$ let $g_n:=f_n+\chi_{F}=\chi_{J_n}+\chi_F,$ and $g:=f+\chi_F=\chi_{K_{1/4}}+\chi_F.$  We claim that there is a solution $v\in BV(\Omega)$ to the Dirichlet problem with boundary data $g$, i.e. $Tv=g$ $\Ha$-a.e.\ on $\partial\Omega.$ 

To show this, we first note that for each $n\in\N$, $g_n$ has a solution, see Theorem~\ref{thm:SeqApprox} proved in Section~4.  Furthermore, $g$ can be extended to a BV function in $\Omega$ by Proposition~\ref{prop:ExtBounds} below, which can then be extended to a BV function on $\R^2.$  Thus by \cite[Lemma~3.1]{LMSS}, this extension of $g$ has a weak solution.  For each $n\in\N,$ we will construct a solution $v_n$ for the Dirichlet problem with boundary data $g_n$, and show that these solutions converge in $L^1(\Omega)$ to a function $v$ whose trace agrees with $g$ $\Ha$-a.e. on $\partial\Omega.$ 

We first denumerate the removed arcs in the construction as follows.  Let $C_{1,1}$ denote the arc removed from $I_0$ in the construction of $f_1$, and similarly, let $C_{2,1}$ and $C_{2,2}$ denote the arcs removed from $I_{1,1}$ and $I_{1,2}$ respectively, in the construction of $f_2.$  Inductively, let $\{C_{n,m}\}_{m=1}^{2^{n-1}}$ be the collection of arcs removed from the arcs $\{I_{n-1,m}\}_{m=1}^{2^{n-1}}$ in the construction of $f_n.$  We recall that each $C_{n,m}$ has arc length $1/4^n,$ and we note that 
$$J_n=I_0\setminus\left(\bigcup_{\ell=1}^n\bigcup_{m=1}^{2^{\ell-1}}C_{\ell,m}\right).$$

For each $n\in\N,$ consider the function $v_n:\Omega\to\R$ given as follows.  For each $\ell\in\{1,\dots,n\}$ and $m\in\{1,\dots, 2^{\ell-1}\},$ let $v_n=0$ in the region of $\Omega$ bounded by $C_{\ell,m}$ and the chord joining the endpoints of $C_{\ell,m},$ and let $v_n=1$ elsewhere in $\Omega$ (see Figure \ref{fig:g_2}).  We wish to show that $v_n$ is a solution to the Dirichlet problem with boundary data $g_n$.

\begin{figure}[h]
\centering
\begin{minipage}[t]{0.45\textwidth}
\centering
\definecolor{ffqqqq}{rgb}{1.,0.,0.}
\definecolor{qqqqff}{rgb}{0.,0.,1.}
\definecolor{uuuuuu}{rgb}{0.26666666666666666,0.26666666666666666,0.26666666666666666}
\begin{tikzpicture}[line cap=round,line join=round,>=triangle 45,x=3.0cm,y=3.0cm]
\clip(-1.2,-1.2) rectangle (1.2,1.2);
\draw [shift={(0.,0.)},line width=0.4pt,color=qqqqff,fill=qqqqff,fill opacity=0.3400000035762787]  plot[domain=1.2217304763960306:1.9198621771937627,variable=\t]({1.*1.0015563522644368*cos(\t r)+0.*1.0015563522644368*sin(\t r)},{0.*1.0015563522644368*cos(\t r)+1.*1.0015563522644368*sin(\t r)});
\draw [shift={(0.,0.)},line width=0.4pt,color=ffqqqq]  plot[domain=-3.6651914291880923:0.523598775598299,variable=\t]({1.*1.001556352264437*cos(\t r)+0.*1.001556352264437*sin(\t r)},{0.*1.001556352264437*cos(\t r)+1.*1.001556352264437*sin(\t r)});
\draw [shift={(0.,0.)},line width=0.4pt,color=qqqqff,fill=qqqqff,fill opacity=0.5299999713897705]  plot[domain=2.181661564992912:2.356194490192345,variable=\t]({1.*1.001556352264437*cos(\t r)+0.*1.001556352264437*sin(\t r)},{0.*1.001556352264437*cos(\t r)+1.*1.001556352264437*sin(\t r)});
\draw [shift={(0.,0.)},line width=0.4pt,color=qqqqff,fill=qqqqff,fill opacity=0.5199999809265137]  plot[domain=0.7853981633974484:0.9599310885968814,variable=\t]({1.*1.0015563522644368*cos(\t r)+0.*1.0015563522644368*sin(\t r)},{0.*1.0015563522644368*cos(\t r)+1.*1.0015563522644368*sin(\t r)});
\draw [shift={(0.,0.)},line width=0.4pt,color=ffqqqq]  plot[domain=2.356194490192345:2.617993877991494,variable=\t]({1.*1.0015563522644368*cos(\t r)+0.*1.0015563522644368*sin(\t r)},{0.*1.0015563522644368*cos(\t r)+1.*1.0015563522644368*sin(\t r)});
\draw [shift={(0.,0.)},line width=0.4pt,color=ffqqqq]  plot[domain=1.9198621771937627:2.181661564992912,variable=\t]({1.*1.0015563522644368*cos(\t r)+0.*1.0015563522644368*sin(\t r)},{0.*1.0015563522644368*cos(\t r)+1.*1.0015563522644368*sin(\t r)});
\draw [shift={(0.,0.)},line width=0.4pt,color=ffqqqq]  plot[domain=0.9599310885968814:1.2217304763960306,variable=\t]({1.*1.001556352264437*cos(\t r)+0.*1.001556352264437*sin(\t r)},{0.*1.001556352264437*cos(\t r)+1.*1.001556352264437*sin(\t r)});
\draw [shift={(0.,0.)},line width=0.4pt,color=ffqqqq]  plot[domain=0.523598775598299:0.7853981633974484,variable=\t]({1.*1.001556352264437*cos(\t r)+0.*1.001556352264437*sin(\t r)},{0.*1.001556352264437*cos(\t r)+1.*1.001556352264437*sin(\t r)});
\draw [line width=1.2pt] (-0.34255244715021665,0.9411551135241432)-- (0.34255244715021665,0.9411551135241432);
\draw [line width=0.4pt] (-0.7082072884266458,0.7082072884266459)-- (-0.5744691233365886,0.8204269334280392);
\draw [line width=0.4pt] (0.5744691233365886,0.8204269334280392)-- (0.7082072884266458,0.7082072884266459);
\begin{scriptsize}
\draw [fill=uuuuuu] (-0.8673732443826784,0.5007781761322185) circle (1.0pt);
\draw [fill=uuuuuu] (0.8673732443826784,0.5007781761322185) circle (1.0pt);
\draw [fill=uuuuuu] (-0.34255244715021665,0.9411551135241432) circle (1.0pt);
\draw [fill=uuuuuu] (0.34255244715021665,0.9411551135241432) circle (1.0pt);
\draw[color=qqqqff] (-0.019172074066183868,1.1033125087737237) node {$C_{1,1}$};
\draw[color=ffqqqq] (0.024746724218267943,-1.1060719355359305) node {$F$};
\draw [fill=uuuuuu] (-0.7082072884266458,0.7082072884266459) circle (1.0pt);
\draw [fill=uuuuuu] (-0.5744691233365886,0.8204269334280392) circle (1.0pt);
\draw [fill=uuuuuu] (0.5744691233365886,0.8204269334280392) circle (1.0pt);
\draw [fill=uuuuuu] (0.7082072884266458,0.7082072884266459) circle (1.0pt);
\draw[color=qqqqff] (-0.7487619067915671,0.8523479471482863) node {$C_{2,1}$};
\draw[color=qqqqff] (0.7014547386011478,0.8648961752295582) node {$C_{2,2}$};
\draw[color=ffqqqq] (-0.8939628317319995,0.6802579620337007) node {$I_{2,1}$};
\draw[color=ffqqqq] (-0.5390272374331645,0.997548872088718) node {$I_{2,2}$};
\draw[color=ffqqqq] (0.502475693312407,0.9814154359842255) node {$I_{2,3}$};
\draw[color=ffqqqq] (0.8717521197041243,0.6677097339524288) node {$I_{2,4}$};
\end{scriptsize}
\end{tikzpicture}
\caption{\small{$v_2=0$ on the shaded region bounding $C_{1,1}$ and on the very small shaded regions bounding $C_{2,1}$ and $C_{2,2},$ and $v_2=1$ elsewhere in the disk.}}\label{fig:g_2}
\end{minipage}
\hfill
\begin{minipage}[t]{0.45\textwidth}
\centering
\definecolor{ffqqqq}{rgb}{1.,0.,0.}
\definecolor{ffffff}{rgb}{1.,1.,1.}
\definecolor{qqqqff}{rgb}{0.,0.,1.}
\begin{tikzpicture}[line cap=round,line join=round,>=triangle 45,x=3.25cm,y=3.25cm]
\clip(-1.1,-0.25) rectangle (1.1,1.5);
\draw [shift={(0.,0.)},line width=0.4pt,color=qqqqff]  plot[domain=1.3365534840106168:1.4611654229983333,variable=\t]({1.*1.*cos(\t r)+0.*1.*sin(\t r)},{0.*1.*cos(\t r)+1.*1.*sin(\t r)});
\draw [shift={(0.,0.)},line width=0.4pt,color=qqqqff]  plot[domain=0.6435011087932843:0.8650756304846283,variable=\t]({1.*1.*cos(\t r)+0.*1.*sin(\t r)},{0.*1.*cos(\t r)+1.*1.*sin(\t r)});
\draw [line width=0.4pt] (-0.8,0.6)-- (0.8,0.6);
\draw [line width=0.4pt] (-0.3921174043245255,0.9199151815389267)-- (0.10925330977028336,0.9925600544486234);
\draw [line width=0.4pt] (0.23210657258891085,0.9726903612974833)-- (0.648582551827898,0.7611443184208971);
\draw [shift={(0.,0.)},line width=0.4pt,color=qqqqff,fill=qqqqff,fill opacity=0.1599999964237213]  plot[domain=0.6435011087932843:2.498091544796509,variable=\t]({1.*1.*cos(\t r)+0.*1.*sin(\t r)},{0.*1.*cos(\t r)+1.*1.*sin(\t r)});
\draw [shift={(0.,0.)},line width=0.4pt,color=ffffff,fill=ffffff,fill opacity=1.0]  plot[domain=0.8650756304846283:1.3365534840106168,variable=\t]({1.*1.*cos(\t r)+0.*1.*sin(\t r)},{0.*1.*cos(\t r)+1.*1.*sin(\t r)});
\draw [shift={(0.,0.)},line width=0.4pt,color=ffffff,fill=ffffff,fill opacity=1.0]  plot[domain=1.4611654229983333:1.9737285321301896,variable=\t]({1.*0.9985548294324231*cos(\t r)+0.*0.9985548294324231*sin(\t r)},{0.*0.9985548294324231*cos(\t r)+1.*0.9985548294324231*sin(\t r)});
\draw [shift={(0.,0.)},line width=0.4pt,color=qqqqff]  plot[domain=1.977354424740988:2.498091544796509,variable=\t]({1.*1.*cos(\t r)+0.*1.*sin(\t r)},{0.*1.*cos(\t r)+1.*1.*sin(\t r)});
\draw [color=qqqqff](-0.110921662800498,0.850709251770478) node[anchor=north west] {$E_0$};
\draw [shift={(0.,0.)},line width=0.4pt]  plot[domain=0.4560418449600677:0.6435011087932843,variable=\t]({1.*1.*cos(\t r)+0.*1.*sin(\t r)},{0.*1.*cos(\t r)+1.*1.*sin(\t r)});
\draw [shift={(0.,0.)},line width=0.4pt]  plot[domain=2.498091544796509:2.6937184671489747,variable=\t]({1.*1.*cos(\t r)+0.*1.*sin(\t r)},{0.*1.*cos(\t r)+1.*1.*sin(\t r)});
\draw [shift={(0.,0.)},line width=0.4pt]  plot[domain=1.4611654229983333:1.977354424740988,variable=\t]({1.*0.9985548294324231*cos(\t r)+0.*0.9985548294324231*sin(\t r)},{0.*0.9985548294324231*cos(\t r)+1.*0.9985548294324231*sin(\t r)});
\draw [shift={(0.,0.)},line width=0.4pt]  plot[domain=0.8650756304846283:1.3365534840106168,variable=\t]({1.*1.*cos(\t r)+0.*1.*sin(\t r)},{0.*1.*cos(\t r)+1.*1.*sin(\t r)});
\draw [line width=0.4pt,dash pattern=on 1pt off 1pt,color=ffqqqq] (0.8,0.6)-- (1.0268602995465377,0.770145224659903);
\draw [shift={(0.,0.)},line width=0.4pt,dash pattern=on 1pt off 1pt,color=ffqqqq]  plot[domain=0.6435011087932843:2.498091544796509,variable=\t]({1.*1.283575374433172*cos(\t r)+0.*1.283575374433172*sin(\t r)},{0.*1.283575374433172*cos(\t r)+1.*1.283575374433172*sin(\t r)});
\draw [line width=0.4pt,dash pattern=on 1pt off 1pt,color=ffqqqq] (-1.0272468560622865,0.7704351420467147)-- (-0.8,0.6);
\begin{scriptsize}
\draw [fill=black] (-0.8,0.6) circle (1.0pt);
\draw [fill=black] (0.8,0.6) circle (1.0pt);
\draw [fill=black] (0.23210657258891085,0.9726903612974833) circle (1.0pt);
\draw [fill=black] (0.10925330977028336,0.9925600544486234) circle (1.0pt);
\draw [fill=black] (0.648582551827898,0.7611443184208971) circle (1.0pt);
\draw[color=qqqqff] (0.8281063736517785,0.784022039968685) node {$C_{l_2,m_2}$};
\draw [fill=black] (-0.3921174043245255,0.9199151815389267) circle (1.0pt);
\draw[color=qqqqff] (-0.7129089260923606,0.8597210371490985) node {$C_{l_1,m_1}$};
\draw[color=ffqqqq] (-0.013594380711394488,1.385108124722683) node {$A$};
\end{scriptsize}
\end{tikzpicture}
\caption{\small{An example of the case when $E_0$ has multiple arcs from $\Cc_n$ in its boundary, and $C_{\ell_1,m_1}$ is the largest sub-arc of $A$ from $\Cc_n$.}}\label{fig:case1}
\end{minipage}
\end{figure}

Let $w\in BV(\Omega)$ be a solution to the Dirichlet problem with boundary data $g_n$, guaranteed to exist by Theorem~\ref{thm:SeqApprox}.  We note that $g_n$ is the characteristic function of a subset of the boundary of $\Omega$, and so by Remark~\ref{remark:SupLevSoln}, we may assume that there exists a set $E\subset\Omega$ such that $w=\chi_E.$  Furthermore we may assume that $\partial E\cap\Omega$ consists of straight line segments.  We will show that multiple arcs from 
$$\Cc_n:=\{C_{\ell,m}:1\le\ell\le n, 1\le m\le 2^{\ell-1}\}$$
cannot be contained in the boundary of a single connected component of $\{w=0\}.$  In doing so, this will show that $w=v_n.$

Suppose that a connected component $E_0$ of $\{w=0\}$ contains multiple arcs from $\Cc_n$ in its boundary.  Let $C_{\ell_1,m_1}$ and $C_{\ell_2,m_2}$ be the extreme arcs joined by $E_0,$ that is, the two arcs connected to $E_0$ which are farthest from one another on $\partial\Omega.$  Let $A$ be the shortest arc on $\partial\Omega$ which contains both $C_{\ell_1,m_1}$ and $C_{\ell_2,m_2}.$   Since the arc $F$ was added to $J_n$ in the construction of $g_n,$ it follows that $w=1$ on the region of $\Omega$ bounded by $F$ and the chord joining the endpoints of $F.$  Therefore, by this choice of $A$, the chord joining the endpoints of $A$ forms part of the perimeter of $E.$  Thus, 
$$\|Dw\|(\Omega)=P(E,\Omega)\ge 2\sin(\Ha(A)/2).$$  
We also note that 
$$\|Dv_n\|(\Omega)<\Ha(I_0\setminus K_{1/4})=1/2.$$  
If $\Ha(A)\ge 5/8$, then 
$$\|Dw\|(\Omega)\ge 2\sin(\Ha(A)/2)\ge 2\sin(5/16)>1/2>\|Dv_n\|(\Omega),$$ contradicting the fact that $w$ is a solution.  Thus, $\Ha(A)<5/8,$ and we have that 
$$\Ha(A)-2\sin(\Ha(A)/2)<0.011.$$

Here is another point at which the addition of $F=\partial\Omega\setminus I_0$ to $K_{1/4}$ makes a difference.  In the previous example, the invalid trace was caused by ``cutting off'' the arcs $I_{k,m}$ in constructing the solution to $f_n.$  When constructing the solution for $g_n$, however, we are unable to ``cut off'' the arc $F$; otherwise, we would have to include the line segment joining the end points of the arc $F$ in the perimeter measure of that solution, creating too much perimeter.

For $k\in\N$ and $m\in\{1,\dots,2^k\},$ we have that 
$$\Ha(I_{k,m}\cap K_{1/4})=\Ha(K_{1/4})/2^k=1/2^{k+1},$$
 and 
$1/2^7<0.011<1/2^6.$  Thus, if $A$ contained a sub-arc $I_{k,m}$ with $k\le 5,$ then we have that 
$$\Ha(A)-2\sin(\Ha(A)/2)<0.011<\Ha(I_{k,m}\cap K_{1/4})\le\Ha(A\cap K_{1/4}),$$ 
hence $\Ha(A\setminus K_{1/4})<2\sin(\Ha(A)/2).$  However, setting $\Omega_A$ to be the open region of $\Omega$ bounded by $A$ and the chord joining the endpoints of $A$, we note that 
$$\|Dv_n\|(\overline\Omega_A\cap\Omega)<\Ha(A\setminus K_{1/4}).$$
  Since the chord joining the endpoints of $A$ has length $2\sin(\Ha(A)/2)$ and comprises part of the perimeter of $E,$ this contradicts the assumption that $w$ is a solution.  Therefore, $A$ cannot contain a sub-arc $I_{k,m}$ with $k\le 5.$  We now consider two cases.  

{\bf Case 1:} Suppose that either $C_{\ell_1,m_1}$ or $C_{\ell_2,m_2}$ is the largest arc in $\Cc_n$ which is a sub-arc of $A$ (see Figure~\ref{fig:case1}). We note that the largest such sub-arc is unique.  Indeed, by the construction, if the arc $A$ contains two arcs $C_{k,m_1}$ and $C_{k,m_2},$ then there exists a $k'< k$ and $1\le m\le 2^{k'}$ such that $C_{k'}\subset A.$  That is to say, $A$ would necessarily contain a larger removed sub-arc. Without loss of generality, we assume that $C_1:=C_{\ell_1,m_1}$ is the largest such sub-arc of $A$.  

Let $B:=A\setminus C_1,$ and let $C_{k,m}$ be the largest arc from $\Cc_n$ such that $C_{k,m}\subset B.$  As above, $C_{k,m}$ is the unique such sub-arc.  We recall that $C_{k,m}$ was removed from the center of the arc $I_{k-1,m}$, and since $C_{k,m}$ was chosen as the largest removed sub-arc of $B,$ it follows from the construction that $B\subset I_{k-1,m}.$  Indeed, if $I_{k-1,m}$ did not contain $B$, then $B$ would contain one of the removed arcs bordering $I_{k-1,m}$, which would necessarily be strictly larger than $C_{k,m}$. It follows from direct computation that 
$$\Ha(B\setminus K_{1/4})\le\Ha(I_{k-1,m}\setminus K_{1/4})=(1+\sum_{j=1}^\infty 1/2^j)/4^k=2/4^k.$$

Moreover, since $C_1$ and $C_{k,m}$ are the largest sub-arcs of $A$ from $\Cc_n$, it follows from the construction that the sub-arc of $B$ connecting $C_1$ to $C_{k,m}$ is of the form $I_{k,m}$ for some $1\le m\le 2^k.$  This is because the arc joining any removed arc $C_{k,m}$ to a larger removed arc, always contains an arc of the form $I_{k,m}$.  Thus, 
$$\Ha(B)\ge\Ha(I_{k,m})=\frac{2^k+1}{2^{2k+1}}>\frac{1}{2^{k+1}}.$$ 
Since $I_{k,m}\subset B\subset A,$ it follows from the prior argument in this proof that $k\ge 6.$ 

We note that $\|Dv_n\|(\overline\Omega_A\cap\Omega)<2\sin(\Ha(C_1)/2)+\Ha(B\setminus K_{1/4}),$ and we would like to show that $$2\sin(\Ha(C_1)/2)+\Ha(B\setminus K_{1/4})<2\sin\left(\frac{\Ha(C_1)+\Ha(B)}{2}\right)=2\sin(\Ha(A)/2).$$
Using the computations above, we prove the following stronger inequality.
\vskip .2cm
\noindent{\it Claim.} 
\begin{equation}\label{eq:SinMeanVal}
\sin\left(\frac{\Ha(C_1)}{2}+\frac{1}{2^{k+2}}\right)-\sin\left(\frac{\Ha(C_1)}{2}\right)>\frac{1}{4^k}.
\end{equation}
\\
\noindent{\it Proof of claim.}  Since $C_1=C_{\ell_1,m_1},$ it follows that $\Ha(C_1)/2\le 1/8.$  By the mean value theorem, there exists $\Ha(C_1)/2<~z_k<\Ha(C_1)/2+1/2^{k+2}$ such that 
$$\sin\left(\frac{\Ha(C_1)}{2}+\frac{1}{2^{k+2}}\right)-\sin\left(\frac{\Ha(C_1)}{2}\right)=\frac{\cos(z_k)}{2^{k+2}}\ge\cos\left(\frac{\Ha(C_1)}{2}+\frac{1}{2^{k+2}}\right)/2^{k+2}.$$ 
Since $\Ha(C_1)/2\le 1/8$ and $k\ge 6,$ it then follows that 
$$\sin\left(\frac{\Ha(C_1)}{2}+\frac{1}{2^{k+2}}\right)-\sin\left(\frac{\Ha(C_1)}{2}\right)\ge 1/2^{k+3}>1/4^k,$$ proving the claim.

\vskip .2cm

Thus, it follows that $\|Dv_n\|(\overline\Omega_A\cap\Omega)<2\sin(\Ha(A)/2)$.  As above, since the chord joining the endpoints of $A$ has length $2\sin(\Ha(A)/2)$ and comprises part of the perimeter of $E,$ this contradicts the assumption that $w$ is a solution.  Therefore,  this configuration cannot occur for $w.$  

\begin{figure}
\centering
\definecolor{ffqqqq}{rgb}{1.,0.,0.}
\definecolor{sqsqsq}{rgb}{0.12549019607843137,0.12549019607843137,0.12549019607843137}
\definecolor{ffffff}{rgb}{1.,1.,1.}
\definecolor{qqqqff}{rgb}{0.,0.,1.}
\begin{tikzpicture}[line cap=round,line join=round,>=triangle 45,x=4.0cm,y=4.0cm]
\clip(-1.,0.4) rectangle (1.,1.4);
\draw [shift={(0.,0.)},line width=0.4pt,color=qqqqff]  plot[domain=1.3365534840106168:1.791500528033564,variable=\t]({1.*1.*cos(\t r)+0.*1.*sin(\t r)},{0.*1.*cos(\t r)+1.*1.*sin(\t r)});
\draw [line width=0.4pt] (-0.8,0.6)-- (0.8,0.6);
\draw [line width=0.4pt] (-0.7068478327999853,0.7073656347787642)-- (-0.21913886929044368,0.9767333373389248);
\draw [line width=0.4pt] (0.23210657258891085,0.9726903612974833)-- (0.7060063818484094,0.7082054707422967);
\draw [shift={(0.,0.)},line width=0.4pt,color=qqqqff,fill=qqqqff,fill opacity=0.1599999964237213]  plot[domain=0.6435011087932843:2.498091544796509,variable=\t]({1.*1.*cos(\t r)+0.*1.*sin(\t r)},{0.*1.*cos(\t r)+1.*1.*sin(\t r)});
\draw [shift={(0.,0.)},line width=0.4pt,color=ffffff,fill=ffffff,fill opacity=1.0]  plot[domain=0.786953154693409:1.3365534840106168,variable=\t]({1.*1.*cos(\t r)+0.*1.*sin(\t r)},{0.*1.*cos(\t r)+1.*1.*sin(\t r)});
\draw [shift={(0.,0.)},line width=0.4pt,color=ffffff,fill=ffffff,fill opacity=1.0]  plot[domain=1.791500528033564:2.3558283488936578,variable=\t]({1.*1.0010144136340535*cos(\t r)+0.*1.0010144136340535*sin(\t r)},{0.*1.0010144136340535*cos(\t r)+1.*1.0010144136340535*sin(\t r)});
\draw [color=qqqqff](-0.08318855167818821,0.8487756647531384) node[anchor=north west] {$E_0$};
\draw [shift={(0.,0.)},line width=0.4pt,color=qqqqff]  plot[domain=1.3365534840106168:1.7942950302170537,variable=\t]({1.*1.*cos(\t r)+0.*1.*sin(\t r)},{0.*1.*cos(\t r)+1.*1.*sin(\t r)});
\draw [shift={(0.,0.)},line width=1.2pt,color=ffqqqq]  plot[domain=1.7942950302170537:2.49662380612191,variable=\t]({1.*1.*cos(\t r)+0.*1.*sin(\t r)},{0.*1.*cos(\t r)+1.*1.*sin(\t r)});
\draw [shift={(0.,0.)},line width=1.2pt,color=ffqqqq]  plot[domain=0.6440880536885113:1.3330657202319052,variable=\t]({1.*1.*cos(\t r)+0.*1.*sin(\t r)},{0.*1.*cos(\t r)+1.*1.*sin(\t r)});
\draw [shift={(0.,0.)},line width=0.4pt]  plot[domain=2.49662380612191:2.689516644503601,variable=\t]({1.*1.*cos(\t r)+0.*1.*sin(\t r)},{0.*1.*cos(\t r)+1.*1.*sin(\t r)});
\draw [shift={(0.,0.)},line width=0.4pt]  plot[domain=0.45347033384214:0.6440880536885113,variable=\t]({1.*1.*cos(\t r)+0.*1.*sin(\t r)},{0.*1.*cos(\t r)+1.*1.*sin(\t r)});
\draw [line width=0.4pt,dash pattern=on 1pt off 1pt,color=ffqqqq] (0.7996476952813643,0.6004694525379308)-- (0.9879233945122827,0.7418489709314736);
\draw [shift={(0.,0.)},line width=0.4pt,dash pattern=on 1pt off 1pt,color=ffqqqq]  plot[domain=0.6440880536885113:2.498091544796509,variable=\t]({1.*1.2354483109773382*cos(\t r)+0.*1.2354483109773382*sin(\t r)},{0.*1.2354483109773382*cos(\t r)+1.*1.2354483109773382*sin(\t r)});
\draw [line width=0.4pt,dash pattern=on 1pt off 1pt,color=ffqqqq] (-0.9883586487818706,0.741268986586403)-- (-0.8,0.6);
\begin{scriptsize}
\draw [fill=black] (-0.8,0.6) circle (1.0pt);
\draw [fill=black] (0.8,0.6) circle (1.0pt);
\draw [fill=black] (0.23210657258891085,0.9726903612974833) circle (1.0pt);
\draw [fill=black] (-0.21913886929044368,0.9767333373389248) circle (1.0pt);
\draw [fill=black] (0.7996476952813643,0.6004694525379308) circle (1.0pt);
\draw [fill=sqsqsq] (-0.2216426561957363,0.9751279572212554) circle (0.5pt);
\draw[color=qqqqff] (-0.008120621363848973,1.0711991619808088) node {$C_3$};
\draw[color=ffqqqq] (-0.63090641360133,0.8932603641986725) node {$B_1$};
\draw [fill=black] (0.23549766819258636,0.9718749139039727) circle (0.5pt);
\draw[color=ffqqqq] (0.644321637170655,0.8765786019065972) node {$B_2$};
\draw[color=ffqqqq] (-0.016461502509886664,1.3279129483644119) node {$A$};
\end{scriptsize}
\end{tikzpicture}
\caption{\small{An example of the case when $E_0$ has multiple arcs from $\Cc_n$ in its boundary, but neither $C_{\ell_1,m_1}$ nor $C_{\ell_2,m_2}$  is the largest sub-arc of $A$ from $\Cc_n$.}}\label{fig:case2}
\end{figure}
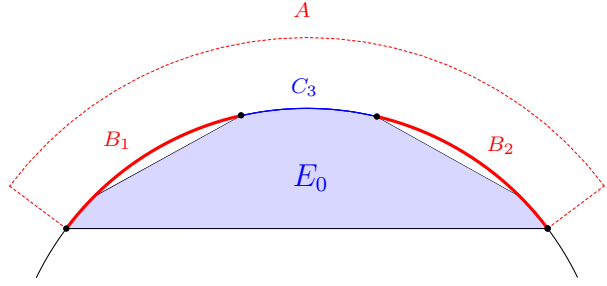

{\bf Case 2:}  Now consider the case when neither $C_{\ell_1,m_1}$ nor $C_{\ell_2,m_2}$ is the largest arc in $\Cc_n$ which is a sub-arc of $A$.  Let $C_3$ denote the largest such sub-arc, and let $B_1$ and $B_2$ be the disjoint arcs which comprise $A\setminus C_3$ (see Figure~\ref{fig:case2}). We iterate the argument from the first case as follows.  First, setting $C_1:=B_1\cup C_3,$ we see that \eqref{eq:SinMeanVal} holds by the same proof.  Since $\Ha(A)<5/8,$ it follows that $\Ha(C_1)/2<3/8$ here, and so for $k\ge 6,$ the same inequalities hold.  Thus, 
$$2\sin(\Ha(A)/2)\ge 2\sin\left(\frac{\Ha(B_1)+\Ha(C_3)}{2}\right)+\Ha(B_2\setminus K_{1/4}).$$  Now, we repeat the argument for \eqref{eq:SinMeanVal}, this time setting $C_1:=C_3.$  Once again, the claim holds, and so it follows that 
$$2\sin\left(\frac{\Ha(B_1)+\Ha(C_3)}{2}\right)\ge 2\sin(\Ha(C_3)/2)+\Ha(B_1\setminus K_{1/4}).$$
Thus, we have that 
$$2\sin(\Ha(A)/2)\ge 2\sin(\Ha(C_3)/2)+\Ha(B_1\setminus K_{1/4})+\Ha(B_2\setminus K_{1/4})>\|Dv_n\|(\overline\Omega_A\cap\Omega),$$
again contradicting the fact that $w$ is a solution.  Having exhausted all the possible cases, it follows that the connected component $E_0$ cannot contain multiple arcs from $\Cc_n$ in its boundary.  Therefore $w=v_n,$ and so $v_n$ solves the Dirichlet problem for boundary data $g_n$.  

Let $v:\Omega\to\R$ be given as follows.  Let $\Cc:=\bigcup_{n\in\N}\Cc_n.$  For each $C\in\Cc,$ let $v=0$ on the region of $\Omega$ bounded by $C$ and the chord joining the endpoints of $C.$  Let $v=1$ elsewhere in $\Omega.$  Then, $v_n\to v$ in $L^1(\Omega),$ and so by \cite[Proposition~3.1]{HKLS}, it follows that $v$ is a function of least gradient in $\Omega.$  If $x\in\partial\Omega\setminus(K_{1/4}\cup F),$ then there exists $C\in\Cc$ such that $x\in C.$  Since each $C$ is an open arc on $\partial\Omega$, it follows that $v=0$ on $B(x,r)\cap\Omega$ for small enough $r>0.$  Thus, $Tv(x)=0=g(x).$  Similarly, if $x\in F,$ and $x$ is not one of the endpoints of $F$, then $Tv(x)=1=g(x).$ 

For $x\in K_{1/4},$ there exists a fixed angle $\theta$ such that $v=1$ on the circular sector of angle $\theta$ in $B(x,r)\cap\Omega$ for sufficiently small $r>0.$  Thus, $Tv(x)>0.$  Since $v$ is the characteristic function of a set in $\Omega,$ $Tv$ must take values of either $1$ or $0$ $\Ha$-a.e., see \cite[Example~4.8]{G1}.  Therefore $Tv(x)=1=g(x)$ for $\Ha$-a.e.\ $x\in K_{1/4},$ and so $Tv=g$ $\Ha$-a.e.\ on $\partial\Omega.$  Hence by Lemma \ref{lem:TraceLeastGradSoln}, $v$ is a solution to the Dirichlet problem with boundary data $g$. 

\end{example}

\begin{remark}\label{rem:NonLin}
Given a domain $\Omega,$ let $\Ss\subset L^1(\partial\Omega)$ denote the set of functions $f$ on $\partial\Omega$ for which there exists a solution to the Dirichlet problem with boundary data $f.$  Examples \ref{ex:GornyKa} and \ref{ex:NonLin} show that $\Ss$ does not necessarily form a vector space, even when $\Omega=B(0,1)\subset\R^2$.  In particular, $\Ss$ can be non-linear.  As shown above, $g$ and $\chi_{F}$ both have solutions, but $g-\chi_F=\chi_{K_{1/4}}=f$ has no solution. 
\end{remark}

\begin{remark}\label{rem:NonLoc}
Example~\ref{ex:NonLin} also shows that solutions to the Dirichlet problem for least gradient functions may be non-local in the following sense.  In Example~\ref{ex:NonLin}, a solution exists for boundary data $g=\chi_{K_{1/4}\cup F}.$  However, for any $I_{k,m}$ from the construction of $K_{1/4}$, let $\eta_{k,m}$ denote a Lipschitz function on $\partial\Omega$ such that $\eta_{k,m}=1$ on $I_{k,m}$ and whose support is contained in the union of $I_{k,m}$ and the two removed arcs adjacent to $I_{k,m}.$  Then there is no solution to the Dirichlet problem with boundary data $\eta_{k,m}g=\chi_{K_{1/4}\cap I_{k,m}}.$  This is because the set $K_{1/4}\cap I_{k,m}$ can be constructed in the same way that $K_{1/4}$ is constructed in Example~\ref{ex:GornyKa}.  By the choice of the parameter $a=1/4,$ the same inequality between the side and base lengths of the trapezoids created in the construction holds at every stage, and so the same argument from Example~\ref{ex:GornyKa} shows that no solution exists for boundary data $\eta_{k,m}g.$  This non-locality contrasts the case involving solutions to the Dirichlet problem for $p$-harmonic functions when $p>1.$  Such solutions are known to exist if the boundary data $f$ is the trace of a function in $N^{1,p}(\Omega).$  If $\eta$ is a compactly supported Lipschitz function on $\partial\Omega$, then $\eta f$ is similarly the trace of a function in $N^{1,p}(\Omega)$ and hence has a solution.  For more on the Dirichlet problem for $p$-harmonic functions and traces of functions of class $N^{1,p}$ when $p>1,$ see for example \cite{BB} and \cite{M}.   
\end{remark}

\section{Sufficient conditions for existence of solutions}

For the remainder of this paper, we assume that $(X,d,\mu)$ is a complete metric measure space supporting a $(1,1)$-Poincar\'e inequality, with $\mu$ a doubling Borel regular measure.  We also assume that $\Omega\subset X$ is a bounded domain, with $\mu(X\setminus\Omega)>0$, such that $\partial\Omega$ has positive mean curvature as in Definition~\ref{defn:MeanCurv}.  Furthermore, we assume that $\Ha(\partial\Omega)<\infty,$ that $\Ha\big|_{\partial\Omega}$ is doubling, and that $\Ha\big|_{\partial\Omega}$ is lower codimension 1 Ahlfors regular, as in \eqref{eq:CodAhlfReg}.  The unit disk in $\R^2$ satisfies the above conditions, and so Example~\ref{ex:GornyKa} above shows that an $L^1$-function on the boundary of such a domain need not have a solution. However, for such domains, it was shown in \cite{LMSS} that given $f\in C(\partial\Omega),$ there exists a solution $u\in BV(\Omega)$ to the Dirichlet problem with boundary data $f$.  We show that the solution for continuous data constructed in \cite{LMSS} is in fact the minimal solution, and that if an arbitrary $f\in L^1(\partial\Omega)$ can be approximated pointwise a.e.\ from above and below by continuous functions, then there exists a solution to the Dirichlet problem with boundary data $f$.  We point out that approximation of the boundary data in this manner was used in \cite{RS2} to show existence of solutions for convex polygonal domains in $\R^2.$  There, the boundary data belonged to $BV(\partial\Omega)$ and satisfied restrictive admissibility conditions with respect to the geometry of the boundary.

Since we are interested in finding solutions to the Dirichlet problem given data $f$ defined on $\partial\Omega,$ we need methods and control over extensions of $f$ into $\Omega$ and to the entire space $X$.  The following results from \cite{MSS} and \cite{LMSS} give existence and bounds for such extensions. 

\begin{prop}\label{prop:ExtBounds}\emph{(\cite[Propositions~4.2, 4.3]{MSS})}  There exists a nonlinear, bounded extension $E:L^1(\partial\Omega)\to BV(\Omega),$ satisfying
\begin{align*}
\|Ef\|_{L^1(\Omega)}&\le C\diam(\Omega)\|f\|_{L^1(\partial\Omega)}\text{ and }\\
 	\|DEf\|(\Omega)&\le C(1+\Ha(\partial\Omega))\|f\|_{L^1(\partial\Omega)}.
\end{align*}
Moreover, $TEf=f$ $\Ha$-a.e.\ on $\partial\Omega$.

\end{prop}

\begin{prop}\label{prop:ExtToX}\emph{(\cite[Lemma~5.1, Proposition 5.2]{LMSS})}  There exists a nonlinear extension $\Ext: C(\partial\Omega)\to C(X)\cap BV(X)$ satisfying 
\begin{align*}
\|\Ext f\|_{L^\infty(X)}&\le\|f\|_{L^\infty(\partial\Omega)}+1\text{ and }\\
\| D\Ext f\|(X)&\le C\left(1+\Ha(\partial\Omega)\right)(\|f\|_{L^1(\partial\Omega)}+\|f\|_{L^\infty(\partial\Omega)}+1).
\end{align*}
Moreover, for $f\in C(\partial\Omega)$ and $z\in\partial\Omega,$ it follows that 
$$\lim_{(X\setminus\partial\Omega)\ni x\to z}\Ext f(x)=f(z).$$
\end{prop} 

We will need the following lemmas.  

\begin{lem}\label{lem:TraceLeastGradSoln}
Let $f\in L^1(\partial\Omega),$ and let $u\in BV(\Omega)$ be a function of least gradient in $\Omega$ such that $Tu=f$ $\Ha$-a.e.\ on $\partial\Omega.$  Then, $u$ is a solution to the Dirichlet problem with boundary data $f$. 
\end{lem}

\begin{proof}
Let $v\in BV(\Omega)$ such that $Tv=f$ $\Ha$-a.e. on $\partial\Omega.$  Then, $T(v-u)=0$ $\Ha$-a.e. on $\partial\Omega,$ and so by \cite[Theorem~6.9]{LS}, there exists $w_k\in BV_c(\Omega)$ for $k\in\N$ such that $w_k\to v-u$ in $BV(\Omega).$  Thus, it follows that 
$$\|D(u+w_k)\|(\Omega)\to\|Dv\|(\Omega),$$ as $k\to\infty.$  Since $u$ is of least gradient and $w_k\in BV_c(\Omega),$ we have that 
$$\|Du\|(\Omega)\le \|D(u+w_k)\|(\Omega)\to\|Dv\|(\Omega)$$ as $k\to\infty$, and so it follows that $\|Du\|(\Omega)\le\|Dv\|(\Omega).$
\end{proof}

\begin{lem}\label{lem:MinMaxSoln}  Let $f,g\in L^1(\partial\Omega)$ be such that $f\le g$ $\Ha$-a.e., and let $u,v\in BV(\Omega)$ be solutions to the Dirichlet problem with boundary data $f$ and $g$ respectively.  Then, $\min\{u,v\}$ and $\max\{u,v\}$ are solutions to the Dirichlet problem with boundary data $f$ and $g$ respectively.
\end{lem}

\begin{proof}
By \cite[Lemma~3.1]{L}, it follows that
\begin{equation}\label{eq:MinMaxEnergy}
\|D\min\{u,v\}\|(\Omega)+\|D\max\{u,v\}\|(\Omega)\le\|Du\|(\Omega)+\|Dv\|(\Omega).
\end{equation}
For $x\in\partial\Omega$ such that $f(x)\le g(x)$ and setting $U_r:=B(x,r)\cap\Omega,$ we have that  
\begin{align*}
\fint_{U_r}|\min\{u,v\}-f(x)|d\mu&=\frac{1}{\mu(U_r)}\left(\int_{U_r\cap\{u\ge v\}}|v-f(x)|d\mu+\int_{U_r\cap\{v>u\}}|u-f(x)|d\mu\right)\\
	&\le\frac{1}{\mu(U_r)}\int_{U_r\cap\{u\ge v\}}|v-f(x)|d\mu+\fint_{U_r}|u-f(x)|d\mu,
\end{align*}
and it follows that 
\begin{align*}
\int_{U_r\cap\{u\ge v\}}&|v-f(x)|d\mu\\
	&=\int_{U_r\cap\{u\ge v\}\cap\{v>f(x)\}}|v-f(x)|d\mu+\int_{U_r\cap\{u\ge v\}\cap\{v\le f(x)\}}|v-f(x)|d\mu\\
	&\le \int_{U_r}|u-f(x)|d\mu+\int_{U_r}|v-g(x)|d\mu.
\end{align*}

Therefore, for $\Ha$-a.e. $x\in\partial\Omega,$ we have that 
$$\fint_{U_r}|\min\{u,v\}-f(x)|d\mu\le 2\fint_{U_r}|u-f(x)|d\mu+\fint_{U_r}|v-g(x)|d\mu\to 0$$ as $r\to 0^+.$  Thus, $T\min\{u,v\}=f$ $\Ha$-a.e.\ on $\partial\Omega,$ and a similar argument shows that $T\max\{u,v\}=g$ $\Ha$-a.e.\ on $\partial\Omega.$ 

Now, suppose that $\|D\min\{u,v\}\|(\Omega)>\|Du\|(\Omega).$  Then by $\eqref{eq:MinMaxEnergy}$ we have that $\|D\max\{u,v\}\|(\Omega)<\|Dv\|(\Omega).$  This is a contradiction since $T\max\{u,v\}=g$ $\Ha$-a.e., and $v$ is a solution to the Dirichlet problem with boundary data $g.$  Thus, we have that
$$\|D\min\{u,v\}\|(\Omega)\le\|Du\|(\Omega),$$ and so $\min\{u,v\}$ is a solution to the Dirichlet problem with boundary data $f.$  Similarly, we have that 
 $\max\{u,v\}$ is a solution to the Dirichlet problem with boundary data $g.$ 
\end{proof}

\begin{lem}\label{lem:SupLevTrace}
Let $f\in L^1(\partial\Omega),$ and suppose that $u\in L^1(\Omega)$ with $Tu=f$ $\Ha$-a.e.\ on $\partial\Omega.$  Then, for $\Leb$-a.e.\ $t\in\R,$ we have that $T\chi_{\{u>t\}}=\chi_{\{f>t\}}$ $\Ha$-a.e.\ on $\partial\Omega.$ 
\end{lem}

\begin{proof}
Since $\Ha(\partial\Omega)<\infty$, it follows that $\Ha(\{f=t\})=0$ for $\Leb$-a.e.\ $t\in\R.$  For such $t\in\R,$ let $x\in\partial\Omega\cap\{f\ne t\}$ be such that $Tu(x)=f(x).$  This property holds for $\Ha$-a.e. $x\in\partial\Omega.$  

We first consider the case when $f(x)>t.$ Setting $U_r:=B(x,r)\cap\Omega,$ we have that 
\begin{align*}
\fint_{U_r}|\chi_{\{u>t\}}-\chi_{\{f>t\}}(x)|d\mu&=\frac{1}{\mu(U_r)}\int_{U_r}|\chi_{\{u>t\}}-1|d\mu\\
	&=\frac{\mu(\{u\le t\}\cap U_r)}{\mu(U_r)}.
\end{align*}
Suppose that this quantity does not go to zero as $r$ goes to zero.  That is, suppose there exists $\eps_0>0$ such that for all $\delta>0,$ there exists $0<r_\delta<\delta$ such that 
$$\frac{\mu(\{u\le t\}\cap U_{r_\delta})}{\mu(U_{r_\delta})}\ge\eps_0.$$

Let $\eps:=|f(x)-t|\eps_0.$  Since $Tu(x)=f(x),$ there exists $\delta>0$ such that 
\begin{align*}
\eps&>\fint_{U_{r_\delta}}|u-f(x)|d\mu\\
	&=\frac{1}{\mu(U_{r_\delta})}\left(\int_{U_{r_\delta}\cap\{u>t\}}|u-f(x)|d\mu+\int_{U_{r_\delta}\cap\{u\le t\}}|u-f(x)|d\mu\right)\\
	&\ge\frac{1}{\mu(U_{r_\delta})}\int_{U_{r_\delta}\cap\{u\le t\}}|u-f(x)|d\mu\\
	&\ge|f(x)-t|\frac{\mu(\{u\le t\}\cap U_{r_\delta})}{\mu(U_{r_\delta})}\ge|f(x)-t|\eps_0=\eps,
\end{align*}
a contradiction.  Thus, it follows that 
$$\fint_{U_r}|\chi_{\{u>t\}}-\chi_{\{f>t\}}(x)|d\mu\to 0$$ as $r\to 0^+.$ 

If $f(x)<t,$ then we have that 
$$\fint_{U_r}|\chi_{\{u>t\}}-\chi_{\{f>t\}}(x)|d\mu=\frac{\mu(\{u>t\}\cap U_r)}{\mu(U_r)}.$$  If we suppose that this quantity does not go to zero as $r$ goes to zero, then we arrive at a contradiction by the same method as above. However, in this case, we use the fact that 
$$\frac{1}{\mu(U_{r_\delta})}\int_{U_{r_\delta}\cap\{u>t\}}|u-f(x)|d\mu\ge|f(x)-t|\frac{\mu(\{u>t\}\cap U_{r_\delta})}{\mu(U_{r_\delta})}.$$
Thus, for $\Ha$-a.e. $x\in\partial\Omega,$ we have that 
$$\lim_{r\to 0}\fint_{U_r}|\chi_{\{u>t\}}-\chi_{\{f>t\}}(x)|d\mu=0,$$ and so $T\chi_{\{u>t\}}=\chi_{\{f>t\}}$ $\Ha$-a.e. on $\partial\Omega.$ 
\end{proof}

\begin{remark}\label{remark:SupLevSoln}
If $u\in BV(\Omega)$ is a solution to the Dirichlet problem with boundary data $f$, then by the coarea formula, we have that $\chi_{\{u>t\}}\in BV(\Omega)$ for $\Leb$-a.e.\ $t\in\R.$  Moreover, as $u$ is necessarily of least gradient in $\Omega$, it follows that for each $t\in\R,$ $\chi_{\{u>t\}}$ is of least gradient in $\Omega$ as shown in \cite[Lemma~3.6]{HKLS}.  Thus, by Lemma \ref{lem:TraceLeastGradSoln} and Lemma \ref{lem:SupLevTrace}, it follows that $\chi_{\{u>t\}}$ is a solution to the Dirichlet problem with boundary data $\chi_{\{f>t\}}$ for $\Leb$-a.e. $t\in\R.$  In particular, if $f=\chi_F$ for some $F\subset\partial\Omega,$ and there exists a solution $u$ to the Dirichlet problem with boundary data $f$, we may assume that $u=\chi_E$ for some $E\subset\Omega$ with $P(E,\Omega)<\infty.$ 
\end{remark}

We now show that the solution constructed in \cite{LMSS} for continuous data is minimal in the following sense.

\begin{prop}\label{prop:MinSoln}  Let $f\in C(\partial\Omega).$  Then, there exists a solution $u\in BV(\Omega)$ to the Dirichlet problem with boundary data $f$, such that for any solution $u'$ to the said Dirichlet problem, we have that $u\le u'$ $\mu$-a.e.\ in $\Omega.$  
\end{prop}

\begin{proof}
We follow the construction of a solution $u\in BV(\Omega)$ given in the proof of \cite[Theorem~4.10]{LMSS} and show that $u$ has the desired property.

Let $\Ext f\in C(X)\cap BV(X)$ be the extension of $f$ to $X$ given by Proposition~\ref{prop:ExtToX}.  For $t\in\R,$ let $F_t:=\{x\in X: \Ext f(x)>t\}.$  Then, $F_t$ is open by continuity of $\Ext f$, and since $\Ext f\in BV(X),$ it follows from the coarea formula that $P(F_t, X)<\infty$ for $\Leb$-a.e.\ $t\in\R.$  Moreover, if $x\in\partial F_t,$ then $\Ext f(x)=t$ by continuity of $\Ext f,$ and so $\partial F_t\cap \partial F_s=\varnothing$ for $s\ne t.$  Since $\Ha(\partial\Omega)<\infty,$ it then follows that $\Ha(\partial\Omega\cap\partial F_t)=0$ for $\Leb$-a.e. $t\in\R.$  Let
$$J:=\{t\in\R: P(F_t,X)<\infty\text{ and }\Ha(\partial\Omega\cap\partial F_t)=0\}.$$

For each $t\in J,$ there exists a unique minimal solution set $\widetilde E_t$ to the Dirichlet problem with boundary data $\chi_{F_t},$ \cite[Proposition~3.7, 4.8]{LMSS}.  Setting 
$$E_t:=\{x\in X:\chi_{\widetilde E_t}^\vee(x)>0\},$$ it follows that $\chi_{E_t}$ is also a minimal solution.  Since $\Leb(\R\setminus J)=0,$ we can find a countable set $I\subset J$ such that $I$ is dense in $\R.$  Let $u:X\to\R$ be given by 
$$u(x)=\sup\{s\in I: x\in E_s\}.$$  In \cite[Theorem~4.10]{LMSS}, it is shown that $u\in BV(\Omega)$ and $u$ is a solution to the Dirichlet problem with boundary data $f$.

Now, let $u'\in BV(\Omega)$ be another solution to the Dirichlet problem with boundary data $f.$  Then, setting $$E'_t=\{u'>t\},$$ it follows from the discussion in Remark \ref{remark:SupLevSoln} that $\chi_{E'_t}$ is a solution to the Dirichlet problem with boundary data $\chi_{F_t}$ for $\Leb$-a.e.\ $t\in\R.$ Letting 
$$J':=\{t\in J:\chi_{E'_t}\text{ is a solution for }\chi_{F_t}\}\subset J,$$ we have that $\Leb(\R\setminus J')=0.$  Hence, there exists a countable set $I'\subset J'$ such that $I'$ is dense in $\R.$  For each $s\in I',$ we have that $\chi_{E_s}\le\chi_{E'_s}$ $\mu$-a.e., since $\chi_{E_s}$ is the minimal solution for $\chi_{F_s}.$  Therefore, letting 
$$G_s:=E_s\setminus E'_s=\left\{x\in \Omega: \chi_{E'_s}(x)<\chi_{E_s}(x)\right\},$$ we have that $\mu(G_s)=0.$  Letting $$G:=\bigcup_{s\in I'} G_s,$$ it follows that $\mu(G)=0.$ 

Moreover, for $s^*\in I'$ and $s\in I$ such that $s^*<s,$ we have from \cite[Lemma~3.8]{LMSS} that $E_s\sqsubset E_{s^*}$.  Let 
$$H_{s^*,s}:=\{x\in\Omega:x\in E_s\setminus E_{s^*}\}.$$  
Then, $\mu(H_{s^*,s})=0.$  Let 
$$H:=\bigcup_{s^*\in I'}\bigcup_{s\in I, s>s^*}H_{s^*,s}.$$  Then, we have $\mu(H)=0.$   

Let $x\in X\setminus (G\cup H),$ and suppose that $u'(x)<u(x)=\sup\{s\in I: x\in E_s\}.$  Then, by the definition of $u$, there exists $s\in I$ with $u'(x)<s\le u(x)$ such that $x\in E_s,$ and by the density of $I'$ in $\R,$ there exists $s^*\in I'$ such that 
$$u'(x)<s^*<s\le u(x).$$  Since $s^*<s$ and $x\not\in H,$ we have that $x\in E_{s^*}.$  Moreover, since $u'(x)<s^*,$ it follows that $x\not\in E'_{s^*}.$  Thus we have that 
$$\chi_{E'_{s^*}}(x)<\chi_{E_{s^*}}(x).$$  However, this is a contradiction since $x\not\in G_{s^*}.$  Therefore, it follows that $u(x)\le u'(x),$ and since $\mu (G\cup H)=0,$ we have that $u\le u'$ $\mu$-a.e. in $\Omega.$
\end{proof}

\begin{remark}
If $v$ is another solution which is minimal in the sense of Proposition~\ref{prop:MinSoln}, then $u=v$ $\mu$-a.e.  Thus, we call $u$ the {\it minimal solution} to the Dirichlet problem with boundary data $f$.
\end{remark}

As an immediate corollary, we obtain a comparison-type result for minimal solutions.

\begin{cor}\label{cor:CompThm}  Let $f,g\in C(\partial\Omega)$ be such that $f\le g$ $\Ha$-a.e.\ on $\partial\Omega$, and let $u$ and $v$ be minimal solutions to the Dirichlet problem with boundary data $f$ and $g$ respectively.  Then $u\le v$ $\mu$-a.e.\ in $\Omega.$  

\end{cor}

\begin{proof}
Existence of the minimal solutions $u$ and $v$ is guaranteed by Proposition~\ref{prop:MinSoln}.  By Lemma \ref{lem:MinMaxSoln}, $\min\{u,v\}$ is a solution to the Dirichlet problem with boundary data $f.$  Therefore, as $u$ is a minimal solution, we have that 
$$u\le\min\{u,v\}\le v$$ 
$\mu$-a.e.\ in $\Omega.$  
\end{proof}  

We note that in this setting, uniqueness of solutions is not guaranteed even for Lipschitz boundary data (see \cite{LMSS}), and so such a comparison theorem may not hold for solutions which are not minimal.  

Using the previous corollary, we are now able to establish Theorem~\ref{thm:SeqApprox}.

\begin{proof}[Proof of Theorem~\ref{thm:SeqApprox}]
For $k\in\N$, let $u_k$ and $v_k$ be the minimal solutions for the Dirichlet problem with boundary data $g_k$ and $h_k$, respectively.  The existence of such minimal solutions follows from Proposition \ref{prop:MinSoln} above.  By Corollary \ref{cor:CompThm}, we have  
\begin{equation}\label{eq:CompIneq}u_k\le u_{k+1}\le v_{k+1}\le v_k\end{equation} $\mu$-a.e.\ in $\Omega.$  Let $u:X\to\R$ be given by $$u:=\sup_{k}u_k=\lim_{k\to\infty}u_k.$$

We have that $|u_k-u|\le 2\max\{|u_1|,|v_1|\},$ $\mu$-a.e. in $\Omega,$ and since $u_1,v_1\in L^1(\Omega),$ it follows that 
$$\int_{\Omega}2\max\{|u_1|,|v_1|\}d\mu<\infty.$$
Therefore, by the dominated convergence theorem, we have that 
$$\lim_{k\to\infty}\int_{\Omega}|u_k-u|d\mu=0,$$ and so it follows that $u_k\to u$ in $L^1(\Omega).$  By \cite[Proposition~3.1]{HKLS}, it follows that $u$ is a function of least gradient in $\Omega.$

Let $Eg_k\in BV(\Omega)$ be the extension of $g_k$ into $\Omega$, as given by Proposition \ref{prop:ExtBounds}.  Since $u_k$ is a solution to the Dirichlet problem with boundary data $g_k$, and since $u_k\to u$ in $L^1(\Omega),$ we have by lower semi-continuity of the BV energy and the bounds from Proposition \ref{prop:ExtBounds} that 
\begin{align*}
\|Du\|(\Omega)\le\liminf_{k\to\infty}\|Du_k\|(\Omega)&\le\liminf_{k\to\infty}\|DEg_k\|(\Omega)\\
	&\le\liminf_{k\to\infty}C(1+\Ha(\partial\Omega))\|g_k\|_{L^1(\partial\Omega)}\\
	&\le C(1+\Ha(\partial\Omega))\|\max\{|g_1|,|h_1|\}\|_{L^1(\partial\Omega)}<\infty.
\end{align*}
Thus, $u\in BV(\Omega).$    

It remains to show that $Tu=f$ $\Ha$-a.e.\ on $\partial\Omega.$  Let $x\in\partial\Omega$ be such that $g_k(x),h_k(x)\to f(x)$ as $k\to\infty,$ and such that for all $k\in\N$, $Tu_k(x)=g_k(x)$ and $Tv_k(x)=h_k(x).$  By the hypothesis, these conditions hold for $\Ha$-a.e.\ $x\in\partial\Omega.$  Then we have that 

\begin{align*}
\fint_{U_r}|u-f(x)|d\mu=\frac{1}{\mu(U_r)}\left(\int_{U_r\cap\{u>f(x)\}}|u-f(x)|d\mu+\int_{U_r\cap\{u\le f(x)\}}|u-f(x)|d\mu\right).
\end{align*}     
Since $\eqref{eq:CompIneq}$ holds for all $k\in\N,$ it follows that $|u(y)-f(x)|\le|v_k(y)-f(x)|$ for $\mu$-a.e. $y\in U_r\cap\{u>f(x)\},$ and similarly, $|u(y)-f(x)|\le|u_k(y)-f(x)|$ for $\mu$-a.e. $y\in U_r\cap\{u\le f(x)\}.$  Therefore we have that 
\begin{align*}
\fint_{U_r}|u-f(x)|d\mu&\le\frac{1}{\mu(U_r)}\left(\int_{U_r\cap\{u>f(x)\}}|v_k-f(x)|d\mu+\int_{U_r\cap\{u\le f(x)\}}|u_k-f(x)|d\mu\right)\\
	&\le\fint_{U_r}|v_k-f(x)|d\mu+\fint_{U_r}|u_k-f(x)|d\mu.
\end{align*}  

Let $\eps>0.$  Since $g_k(x),h_k(x)\to f(x)$ as $k\to\infty,$ we can choose $k\in\N$ sufficiently large such that $|g_k(x)-f(x)|,|h_k(x)-f(x)|<\eps.$  Thus, it follows that 
\begin{align*}
\fint_{U_r}|u-&f(x)|d\mu\\\
	&\le\fint_{U_r}|v_k-h_k(x)|+|h_k(x)-f(x)|d\mu+\fint_{U_r}|u_k-g_k(x)|+|g_k(x)-f(x)|d\mu\\
	&<\fint_{U_r}|v_k-h_k(x)|d\mu+\fint_{U_r}|u_k-g_k(x)|d\mu+2\eps.
\end{align*}
Since $Tu_k(x)=g_k(x)$ and $Tv_k(x)=h_k(x),$ we have that 
$$\lim_{r\to 0^+}\fint_{U_r}|u-f(x)|d\mu<2\eps,$$ and so taking $\eps\to 0,$ it follows that $Tu(x)=f(x).$  Thus, $Tu=f$ $\Ha$-a.e.\ on $\partial\Omega,$ and so by Lemma \ref{lem:TraceLeastGradSoln}, $u\in BV(\Omega)$ is a solution to the Dirichlet problem with boundary data $f$. 

Let $u'\in BV(\Omega)$ be a solution to the Dirichlet problem with boundary data $f.$  For each $k\in\N,$ it follows that $\min\{u_k, u'\}$ is a solution to the Dirichlet problem with boundary data $g_k$, by Lemma \ref{lem:MinMaxSoln}.  However, as $u_k$ is a minimal solution, it follows that $u_k\le\min\{u_k,u'\}\le u'$ $\mu$-a.e.\ in $\Omega.$  Since $u_k\to u$ pointwise $\mu$-a.e., it follows that $u\le u'$ $\mu$-a.e.\ in $\Omega.$  Thus $u$ is a minimal solution, and uniqueness follows from minimality.   
\end{proof}

\begin{example}\label{ex:NotConverge}
We point out that the functions $v_k$ above need not converge to the minimal solution $u.$  Let $\Omega\subset\R^2$ be the (unweighted) unit disc, and let 
$$f(x,y)=
\begin{cases}
1,&|y|>1/\sqrt 2\\
0,&|y|\le 1/\sqrt 2.
\end{cases}
$$
Letting
\begin{align*}
g_k(x,y)&=
\begin{cases}
1,&|y|\ge 1/\sqrt 2+2/k\\
k(y-(1/\sqrt 2+1/k)),& 1/\sqrt 2+1/k<y<1/\sqrt 2+2/k\\
-k(y+1/\sqrt 2+1/k),& -(1/\sqrt 2+2/k)<y<-(1/\sqrt 2+1/k)\\
0,&|y|\le 1\sqrt 2+1/k
\end{cases}
\end{align*}
\text{and}
\begin{align*}
h_k(x,y)&=
\begin{cases}
1,&|y|\ge 1/\sqrt 2-1/k\\
k(y-(1/\sqrt 2-2/k)),&1/\sqrt 2<y<1/\sqrt 2-1/k\\
-k(y+1/\sqrt 2-2/k),&-(1/\sqrt 2-1/k)<y<-(1/\sqrt 2-2/k)\\
0,&|y|\le 1\sqrt 2-2/k,
\end{cases}
\end{align*}
for sufficiently large $k$, we have that $g_k\le f\le h_k.$  It follows that the $u_k$ converge to the minimal solution $u=\chi_{\Omega\cap\{|y|>1/\sqrt 2\}}$ to the Dirichlet problem with boundary data $f$, and the $v_k$ converge to the maximal solution $v=\chi_{\Omega\cap\{|x|<1/\sqrt 2\}}.$  For more on non-uniqueness of solutions, see \cite[Example~2.7]{MRL} and \cite{Gnon}.
\end{example}

As a consequence of Theorem~\ref{thm:SeqApprox}, we are able to prove Theorem~\ref{thm:HZeroBoundary}.  Recall that given $F\subset\partial\Omega,$ we let $\widetilde\partial F$ denote the boundary of $F$ {\it relative} to $\partial\Omega.$

\begin{proof}[Proof of Theorem~\ref{thm:HZeroBoundary}]
For $\eps>0,$ let $\eta^+_\eps:\partial\Omega\to\R$ be given by 
$$\eta^+_\eps(x)=\max\left\{1-\frac{\dist(x,F)}{\eps},0\right\}.$$
Then $\eta_\eps^+$ is continuous, and $\chi_F\le\eta^+_{\eps_1}\le\eta_{\eps_2}^+$ on $\partial\Omega$ for $\eps_1<\eps_2.$  If $x\in F,$ then $\eta_\eps^+(x)=1$ for all $\eps>0.$  Likewise, if $x\in\partial\Omega$ is in the exterior of $F$ relative to $\partial\Omega,$ then $\eta_\eps^+(x)=0$ for sufficiently small $\eps>0.$  Therefore, $\lim_{\eps\to 0}\eta_\eps^+(x)=\chi_F(x)$ for all $x\in\partial\Omega\setminus\wtil\partial F,$ hence $\Ha$-a.e.\ on $\partial\Omega.$ 

Similarly, for $\eps>0,$ let $\eta_\eps^-:\partial\Omega\to\R$ be given by 
$$\eta_\eps^-(x)=\min\left\{\frac{\dist(x,\Omega\setminus F)}{\eps},1\right\}.$$
Then $\eta_\eps^-$ is continuous, and $\eta_{\eps_2}^-\le\eta_{\eps_1}^-\le\chi_F$ on $\partial\Omega$ for $\eps_1<\eps_2.$  Similarly since $\Ha(\wtil\partial F)=0,$ we have that $\eta_\eps^-\to\chi_F$ $\Ha$-a.e.\ on $\partial\Omega$ as $\eps\to 0.$  Having constructed the necessary approximating sequences, Theorem~\ref{thm:SeqApprox} gives the existence of the minimal solution $u\in BV(\Omega).$   
\end{proof}

We note that existence of a solution in Theorem~\ref{thm:HZeroBoundary} also follows from the results of Section 5.  However, minimality of the solution is not guaranteed by those results.  

As the following example shows, some sets $F\subset\partial\Omega$ with poorly behaved boundaries still satisfy the $\Ha(\wtil\partial F)=0$ condition.  Theorem~\ref{thm:HZeroBoundary} gives us a way to ensure existence of minimal solutions even in these cases.

\begin{example}
Let $\Omega=B(0,1)\subset\R^3,$ and let $K\subset\partial\Omega=S^2$ be a bi-Lipschitz embedding of the von Koch snowflake domain in $\R^2$ onto $S^2.$  Then, $\wtil\partial K$ is a curve of infinite length, and so $\chi_K\in L^1(\partial\Omega)\setminus BV(\partial\Omega).$   Since $\Ha\simeq\Ha^2$ in this example, where $\Ha^2$ is the standard 2-dimensional Hausdorff measure, it follows that $\Ha(\wtil\partial K)=0$.  Hence by Theorem~\ref{thm:HZeroBoundary}, there exists a minimal solution to the Dirichlet problem with boundary data $\chi_K.$  

\end{example}

\begin{remark}\label{rem:NotSharp1}  It should be noted that the conditions imposed on $f$ in Theorem~\ref{thm:SeqApprox} are rather strict.  Unbounded functions are excluded, for example. Furthermore, the conditions are not necessary in general to guarantee existence of a solution.  As seen in Example~\ref{ex:NonLin} above, the function $g$ has a solution but cannot be approximated $\Ha$-a.e.\ from below by continuous functions.  Likewise, Example~\ref{ex:NonLin} also shows that the conditions imposed on a subset of $\partial\Omega$ in Theorem~\ref{thm:HZeroBoundary} are not necessary in general.  Since $g=\chi_{K_{1/4}\cup F},$ we have that $\Ha(\wtil\partial(K_{1/4}\cup F))=\Ha(\wtil\partial K_{1/4})=\Ha(K_{1/4})=1/2.$   
\end{remark}

\section{Solutions at points of continuity of the boundary data}

In this section, we consider an arbitrary $f\in L^1(\partial\Omega),$ and show that there exists a least gradient function $u\in BV(\Omega)$ whose trace agrees with $f$ at points of continuity of $f$.  The argument follows that of \cite[Theorem~3.1]{G2}, where an analogous result was shown for strictly convex domains in $\R^n.$  However, our results apply even for domains that are not strictly convex (see Remark~\ref{rem:Cylinder}).  We first need the following lemmas regarding minimal solution sets.

\begin{lem}\label{lem:DisjtSoln}  Let $F_1,F_2\subset X$ be open sets such that $P(F_1,X),P(F_2,X)<\infty$ and $\Ha(\partial F_1\cap\partial\Omega)=0=\Ha(\partial F_2\cap\partial\Omega).$  Suppose also that $F_1\cap F_2\cap\partial\Omega=\varnothing,$ and let $E_1, E_2\subset~X$ be minimal solution sets to the Dirichlet problem with boundary data $\chi_{F_1}$ and $\chi_{F_2}$ respectively.  Then, $\mu(E_1\cap E_2\cap\Omega)=0.$  
\end{lem}

\begin{proof}  Let $\partial\Omega^*:=\{x\in\partial\Omega: T\chi_{E_1}(x)=\chi_{F_1}(x)\text{ and }T\chi_{E_2}(x)=\chi_{F_2}(x)\}.$  We claim that $T\chi_{E_1\setminus E_2}(x)=\chi_{F_1}(x)$ for $x\in\partial\Omega^*,$ and hence for $\Ha$-a.e.\ $x\in\partial\Omega.$  

Indeed,  for $x\in\partial\Omega^*\cap F_1,$ (setting $U_r:=B(x,r)\cap\Omega)$, we have that 
\begin{align*}
\fint_{U_r}|\chi_{E_1\setminus E_2}-&\chi_{F_1}(x)|d\mu\\
&=\frac{1}{\mu(U_r)}\left(\int_{U_r\cap E_2}|\chi_{E_1\setminus E_2}-1|d\mu+\int_{U_r\setminus(E_1\cup E_2)}|\chi_{E_1\setminus E_2}-1|d\mu\right)\\
	&=\frac{\mu(U_r\cap E_2)}{\mu(U_r)}+\frac{\mu(U_r\setminus(E_1\cup E_2))}{\mu(U_r)}\\
	&\le\frac{\mu(U_r\cap E_2)}{\mu(U_r)}+\frac{\mu(U_r\setminus E_1)}{\mu(U_r)}\\
	&=\fint_{U_r}|\chi_{E_2}-\chi_{F_2}(x)|d\mu+\fint_{U_r}|\chi_{E_1}-\chi_{F_1}(x)|d\mu\to 0
\end{align*} 
as $r\to 0^+.$  Similarly, for $x\in\partial\Omega^*\setminus F_1,$ we have that 
\begin{align*}
\fint_{U_r}|\chi_{E_1\setminus E_2}-\chi_{F_1}(x)|d\mu=\fint_{U_r}\chi_{E_1\setminus E_2}d\mu\le\fint_{U_r}\chi_{E_1}d\mu=\fint_{U_r}|\chi_{E_1}-\chi_{F_1}(x)|d\mu\to 0
\end{align*}
as $r\to 0^+.$  Thus, we have that $T\chi_{E_1\setminus E_2}=\chi_{F_1}$ $\Ha$-a.e.\ on $\partial\Omega.$  Likewise, a symmetric argument shows that $T\chi_{E_2\setminus E_1}=\chi_{F_2}$ $\Ha$-a.e.\ on $\partial\Omega.$

Now, by \cite[Proposition~4.7]{MMJr}, we have that 
\begin{align*}
P(E_1\setminus(E_2\cap\Omega),\Omega)&=P(E_1\setminus E_2,\Omega)\\
	&=P(E_1\cap(\Omega\setminus E_2),\Omega)\\
	&\le P(E_1,\Omega)+P(\Omega\setminus E_2,\Omega)-P(E_1\cup(\Omega\setminus E_2),\Omega)\\
	&=P(E_1,\Omega)+P(E_2,\Omega)-P(E_2\setminus E_1,\Omega)\\
	&=P(E_1,\Omega)+P(E_2,\Omega)-P(E_2\setminus (E_1\cap\Omega),\Omega).
\end{align*}
Thus, we have that 
$$P(E_1\setminus(E_2\cap\Omega),\Omega)+P(E_2\setminus (E_1\cap\Omega),\Omega)\le P(E_1,\Omega)+P(E_2,\Omega).$$
If $P(E_1\setminus (E_2\cap\Omega),\Omega)>P(E_1,\Omega),$ then we have that $P(E_2\setminus (E_1\cap\Omega),\Omega)<P(E_2,\Omega)$.  However, since $T\chi_{E_2\setminus E_1}=\chi_{F_2}$ $\Ha$-a.e.\ on $\partial\Omega,$ and since $E_2$ is a solution set for $\chi_{F_2},$ this is a contradiction.  Therefore, we have that $P(E_1\setminus (E_2\cap\Omega),\Omega)\le P(E_1,\Omega),$ and so $E_1\setminus (E_2\cap\Omega)$ is a solution set for $\chi_{F_1}.$ Thus, since $E_1$ is the minimal solution set, we have that $\mu(E_1\cap E_2\cap\Omega)=0.$     
\end{proof}

The following lemma follows from a similar argument.

\begin{lem}\label{lem:SubsetSoln}
Let $F_1,F_2\subset X$ be open sets such that $P(F_1,X), P(F_2,X)<\infty$ and $\Ha(\partial F_1\cap\partial\Omega)=0=\Ha(\partial F_2\cap\partial\Omega).$  Suppose also that $F_1\cap\partial\Omega\subset F_2,$ and let $E_1, E_2\subset X$ be minimal solution sets to the Dirichlet problem with boundary data $\chi_{F_1}$ and $\chi_{F_2}$ respectively.  Then, $E_1\cap\Omega\sqsubset E_2.$  
\end{lem}

Since $X$ is doubling, it follows that for each $K\ge 1,$ there exists $C_K>0$ such that for every $r>0,$ we can find a finite cover $\{B(x_i,r)\}_{i\in J_r}\subset X$ of $\partial\Omega$ with $x_i\in\partial\Omega$ such that $\sum_{i\in J_r}\chi_{B(x_i,Kr)}\le C_K$.  Let $\eps>0,$ and consider such a cover $\{B(x_i,\eps/5)\}_{i\in J_\eps}.$  By the 5B-Lemma, there exists a disjoint subcollection $\{B(x_i,\eps/5)\}_{i\in I_\eps\subset J_\eps}$ such that $\partial\Omega\subset\bigcup_{i\in I_\eps}B(x_i,\eps).$  Thus, we obtain a finite cover $\{B_{i,\eps}:=B(x_i,\eps)\}_{i\in I_\eps}$ of $\partial\Omega$ such that the set $\{x_i\}_{i\in I_\eps}$ is $2\eps/5$-separated, and for all $K\ge 1,$ we have that 
\begin{equation}\label{BoundedOverlap}
\sum_{i\in I_\eps}\chi_{KB_{i,\eps}}\le C_K.
\end{equation} 
We can then find a Lipschitz partition of unity $\{\phii_i^\eps\}_{i\in I_\eps}$ subject to this cover.  That is, for each $i\in I_\eps,$ there is a $C/\eps$-Lipschitz function $\phii_i^\eps:X\to[0,1],$ (with $C$ depending only on the doubling constant) such that $\overline{\supt(\phii_i^\eps)}\subset 2B_{i,\eps},$ and $\sum_i\phii_i^\eps=1$ on $\partial\Omega.$  For proof of these facts, see for example \cite[Appendix~B]{GR}, \cite{H}, and \cite{HKST}.

For $f\in L^1(\partial\Omega),$ define
$$f_{B_{i,\eps}}:=\fint_{B_{i,\eps}\cap\partial\Omega}f\,d\Ha,$$ and let $f_\eps:\partial\Omega\to\R$ be given by 
$$f_\eps:=\sum_{i\in I_\eps} f_{B_{i,\eps}}\phii_i^\eps\big|_{\partial\Omega}.$$  Then, as $f_\eps$ is continuous, there is a minimal solution $u_\eps\in BV(\Omega)$ to the Dirichlet problem with boundary data $f_\eps,$ by Proposition \ref{prop:MinSoln}. 

\begin{lem}\label{lem:DiscConvBound}
Let $u_\eps$ be the minimal solution to the Dirichlet problem with boundary data $f_\eps$, as defined above.  Then, 
$$\sup_{\eps>0}\left(\|u_\eps\|_{L^1(\Omega)}+\|Du_\eps\|(\Omega)\right)<\infty.$$  That is, $\{u_\eps\}_{\eps>0}$ is bounded in $BV(\Omega).$ 
\end{lem}  

\begin{proof}  We have that
\begin{align*}
\int_{\partial\Omega}|f_\eps|d\Ha&\le\sum_{i\in I_\eps}\int_{\partial\Omega}|f_{B_{i,\eps}}\phii_{i}^\eps|d\Ha\\
	&\le\sum_{i\in I_\eps}\sum_{j\in I_\eps}\int_{B_{j,\eps}\cap\partial\Omega}|f_{B_{i,\eps}}\phii_i^\eps|d\Ha.
\end{align*}
For $i\in I_\eps,$ let $J_{i,\eps}:=\{j\in I_\eps: B_{j,\eps}\cap 2B_{i,\eps}\ne\varnothing\}.$  If $j\in J_{i,\eps},$ then $B_{j,\eps}\subset 4B_{i,\eps},$ and so by the doubling property and the fact that $\{x_i\}_{i\in I_\eps}$ is a $2\eps/5$-separated set, there exists $C>0$ depending only on the doubling constant such that $|J_{i,\eps}|\le C,$ where $|J_{i,\eps}|$ denotes the number of elements in $J_{i,\eps}.$  Since $\phii_i^\eps$ is compactly supported in $2B_{i,\eps},$ it follows that $\phii_i^\eps=0$ on $B_{j,\eps}$ for $j\in I_\eps\setminus J_{i,\eps}.$  Therefore, we have that 
\begin{align}
\int_{\partial\Omega}|f_\eps|d\Ha&\le\sum_{i\in I_\eps}\sum_{j\in J_{i,\eps}}\int_{B_{j,\eps}\cap\partial\Omega}|f_{B_{i,\eps}}\phii_i^\eps|d\Ha\nonumber\\
	&\le\sum_{i\in I_\eps}\sum_{j\in J_{i,\eps}}\int_{B_{j,\eps}\cap\partial\Omega}\left(\fint_{B_{i,\eps}\cap\partial\Omega}|f|d\Ha\right)d\Ha\nonumber\\
	&=\sum_{i\in I_\eps}\sum_{j\in J_{i,\eps}}\frac{\Ha(B_{j,\eps}\cap\partial\Omega)}{\Ha(B_{i,\eps}\cap\partial\Omega)}\int_{B_{i,\eps}\cap\partial\Omega}|f|d\Ha\nonumber\\
	&\le C\sum_{i\in I_\eps}\int_{B_{i,\eps}\cap\partial\Omega}|f|d\Ha\le C\int_{\partial\Omega}|f|d\Ha,\label{f_epsBound}
\end{align}
where the constant $C>0$ depends on the doubling constant of $\Ha$ and the bounded overlap constant from \eqref{BoundedOverlap}.

Let $Ef_\eps:\Omega\to\R$ be the extension of $f_\eps$ to $\Omega$ given by Proposition $\ref{prop:ExtBounds}$.  Since $u_\eps\in BV(\Omega)$ is a solution to the Dirichlet problem with boundary data $f_\eps,$ and by ($\ref{f_epsBound}$), we have that 
\begin{equation}\label{eq:EnergyBound}
\|Du_\eps\|(\Omega)\le\|DEf_\eps\|(\Omega)\le C\int_{\partial\Omega}|f_\eps|d\Ha\le C\int_{\partial\Omega}|f|d\Ha.
\end{equation}

Now, let $v_\eps:\Omega\to\R$ be given by $v_\eps:=u_\eps-Ef_\eps,$ and let $\hat v_\eps$ be the zero extension of $v_\eps$ to all of $X.$  Since $Tu_\eps=TEf_\eps=f_\eps$ $\Ha$-a.e.\ on $\partial\Omega,$ it follows that 
$$\lim_{r\to 0^+}\frac{1}{\mu(B(x,r))}\int_{B(x,r)\cap\Omega}|v_\eps|d\mu=0$$ for $\Ha$-a.e.\ $x\in\partial\Omega.$  Thus, by \cite[Theorem~6.1]{LS}, it follows that $\hat v_\eps\in BV(X)$ and $\|D\hat v_\eps\|(X\setminus\Omega)=0.$

  Since $\mu(X\setminus\Omega)>0,$ we can find a ball $B\subset X$ such that $\Omega\subset B$ and $\mu(B\setminus\Omega)>0.$  By H\"older's Inequality, we have that 
\begin{align*}\int_\Omega|v_\eps|d\mu\le\int_B|\hat v_\eps|d\mu&\le\mu(B)^{1/Q}\left(\int_B|\hat v_\eps|^{\frac{Q}{Q-1}}d\mu\right)^{\frac{Q-1}{Q}}\\
	&=\mu(B)\left(\fint_B|\hat v_\eps|^{\frac{Q}{Q-1}}d\mu\right)^{\frac{Q-1}{Q}},
\end{align*}
where $Q>1$ is the exponent from \eqref{eq:LMBExp}.  Since $\hat v_\eps=0$ on $B\setminus\Omega,$ it follows from Lemma $\ref{lem:MazyaIneq}$ that 
\begin{align*}
\int_\Omega|v_\eps|d\mu&\le\mu(B)\frac{C\rad(B)}{1-(\mu(\Omega)/\mu(B))^{1/Q}}\frac{\|D\hat v_\eps\|(2\lambda B)}{\mu(2\lambda B)}\\
	&\le C\left(\|D\hat v_\eps\|(\Omega)+\|D\hat v_\eps\|(X\setminus\Omega)\right)\\
	&=C\|D v_\eps\|(\Omega)\\
	&\le C\left(\|Du_\eps\|(\Omega)+\|DEf_\eps\|(\Omega)\right).
\end{align*}
Here the constant $C>0$ depends on $\Omega$, $B$, and the doubling and Poincar\'e constants, but is independent of $\eps,$ $f_\eps,$ and $u_\eps.$  

Then by $\eqref{eq:EnergyBound},$ we have that 
$$\int_\Omega|v_\eps|d\mu\le C\int_{\partial\Omega}|f|d\Ha,$$ and by the Triangle Inequality, it follows that 
$$\int_\Omega|u_\eps|d\mu\le C\int_{\partial\Omega}|f|d\Ha+\int_\Omega|Ef_\eps|d\mu.$$  Therefore, by Proposition $\ref{prop:ExtBounds}$ and ($\ref{f_epsBound}$), we have that 
$$\int_{\Omega}|u_\eps|d\mu\le C\int_{\partial\Omega}|f|d\Ha.$$
  
\end{proof}

\begin{lem}\label{lem:UnifCty}
Let $f$ be continuous at $x\in\partial\Omega.$  Then, for all $\eta>0$ there exists $\delta>0$ such that for all $0<\eps<\delta$ and for all $y\in B(x,\delta)\cap\partial\Omega,$ we have that $|f_\eps(y)-f(x)|<\eta.$ 
\end{lem}

\begin{proof}  Let $\eta>0.$  By the continuity of $f$ at $x,$ there exists $\delta_0>0$ such that if $y\in B(x,\delta_0)\cap\partial\Omega,$ then $|f(y)-f(x)|<\eta.$  Let $\delta:=\delta_0/10.$  Then, for $y\in B(x,\delta)\cap\partial\Omega,$ and for $0<\eps<\delta,$ we have that 
\begin{align*}
|f_\eps(y)-f(x)|&=\left|\sum_{i\in I_\eps}f_{B_{i,\eps}}\phii_i^\eps(y)-f(x)\sum_{i\in I_\eps}\phii_i^\eps(y)\right|\\
	&\le\sum_{i\in I_\eps}\phii_i^\eps(y)|f_{B_{i,\eps}}-f(x)|.  
\end{align*} 
Let $J_{y,\eps}:=\{i\in I_\eps: y\in 2B_{i,\eps}\}.$  Then, $\phii_i^\eps(y)=0$ for all $i\in I_\eps\setminus J_{y,\eps}.$  Thus, we have that 
$$|f_\eps(y)-f(x)|\le\sum_{i\in J_{y,\eps}}\phii_i^\eps(y)|f_{B_{i,\eps}}-f(x)|.$$  
Since $0<\eps<\delta,$ it follows that for $i\in J_{y,\eps},$ we have that $B_{i,\eps}\subset B(x,\delta).$  Therefore, 
$$|f_{B_{i,\eps}}-f(x)|\le\fint_{B_{i,\eps}\cap\partial\Omega}|f-f(x)|d\Ha<\eta,$$ and so it follows that $|f_\eps(y)-f(x)|<\eta.$

\end{proof}

We are now able to establish Theorem~\ref{thm:CtyPointSoln}.

\begin{proof}[Proof of Theorem~\ref{thm:CtyPointSoln}]  Since $(u_\eps)_{\eps>0}$ is bounded in $BV(\Omega),$ it follows from the compact embedding theorem \cite[Theorem~3.7]{MMJr} that there exists $u\in BV_{loc}(\Omega)$ and a subsequence, also denoted $u_\eps\in BV(\Omega),$ such that $u_\eps\to u$ in $L^1_{loc}(\Omega),$ and passing to a further subsequence if necessary, we have that $u_\eps\to u$ pointwise a.e.\ in $\Omega.$  Hence by Fatou's lemma and Lemma~\ref{lem:DiscConvBound}, we have that 
$$\int_{\Omega}|u|d\mu\le\liminf_{\eps\to 0}\int_{\Omega}|u_{\eps}|d\mu<\infty.$$  By lower semi-continuity of the BV energy and $\eqref{eq:EnergyBound},$ we have that $\|Du\|(\Omega)<\infty$, and so $u\in BV(\Omega).$  Furthermore by \cite[Proposition~3.1]{HKLS}, it follows that $u$ is a function of least gradient.

Let $f$ be continuous at $x\in\partial\Omega,$ and let $\eta>0.$  By Lemma \ref{lem:UnifCty}, there exists $\delta>0$ such that for all $0<\eps<\delta,$ and for all $y\in B(x,\delta)\cap\partial\Omega,$ we have that $$|f_\eps(y)-f(x)|<\eta.$$  

Let $r_x>0$ be as in Definition \ref{defn:MeanCurv}.  Then by the coarea formula, and since $\Ha(\partial\Omega)<\infty,$ there exists $\delta_x>0$ such that 
$$\min\{\delta,r_x\}/2<\delta_x<\min\{\delta,r_x\},$$ with $P(B(x,\delta_x),X)<\infty$ and $\Ha(\partial B(x,\delta_x)\cap\partial\Omega)=0.$  Denote $F_\eta:=B(x,\delta_x),$ and let $E_\eta\subset X$ be a minimal solution set for $\chi_{F_\eta}.$  Let $\widehat\delta:=\min\{\delta_x,\phi_x(\delta_x)\},$ where $\phi_x$ is as in Definition \ref{defn:MeanCurv}.  Then, for all $0<r<\widehat\delta,$ it follows that $B(x,r)\cap\Omega\sqsubset E_\eta.$

For $0<\eps<\widehat\delta,$ let $F^\eps_t:=\{\Ext f_\eps>t\},$ where $\Ext f_\eps$ is the extension of $f_\eps$ to $X$ given by Proposition \ref{prop:ExtToX}.  Recall that $u_\eps:X\to\R$ is given by 
$$u_\eps(y)=\sup\{t\in\mathcal I_\eps:y\in E^\eps_t\},$$ 
where $\mathcal I_\eps$ and $E^\eps_t$ are as in the proof of Proposition \ref{prop:MinSoln}.  By the choice of $\widehat\delta,$ we note that for $t\in\mathcal I_\eps$ such that $t\ge f(x)+\eta,$ it follows that $F_t^\eps\cap F_\eta\cap\partial\Omega=\varnothing.$  Thus, by Lemma \ref{lem:DisjtSoln}, we have that $\mu(E_t^\eps\cap E_\eta\cap\Omega)=0.$  Similarly, for $t\in\mathcal I_\eps$ such that $t\le f(x)-\eta,$ we have that $F_\eta\cap\partial\Omega\subset F_t^\eps,$ and so by Lemma \ref{lem:SubsetSoln}, it follows that $E_\eta\cap\Omega\sqsubset E_t^\eps.$  Therefore, for all $0<\eps<\widehat\delta$ and for all $0<r<\widehat\delta,$ it follows from the construction of $u_\eps$ that 
$$|u_\eps(y)-f(x)|\le\eta$$ for $\mu$-a.e.\ $y\in B(x,r)\cap\Omega.$ 

Therefore, for all $\eta>0,$ we have that 
$$\lim_{r\to 0^+}\lim_{\eps\to 0^+}\fint_{B(x,r)\cap\Omega}|u_\eps-f(x)|d\mu\le\eta,$$ and as $u_\eps\to u$ pointwise a.e., it follows from the dominated convergence theorem that 
$$\lim_{r\to 0^+}\fint_{B(x,r)\cap\Omega}|u-f(x)|d\mu=0.$$  Thus, we have that $Tu(x)=f(x).$               

\end{proof}

We note that in the proof of Theorem~\ref{thm:CtyPointSoln}, we apply Lemmas~\ref{lem:DisjtSoln} and \ref{lem:SubsetSoln} to the regularized sets 
$$E^\eps_t=\{x\in X:\chi^\vee_{\widetilde E^\eps_t}(x)>0\},$$
where $\widetilde E^\eps_t$ is the minimal solution set for $\chi_{F^\eps_t}.$  However since $\mu(E^\eps_t\Delta\widetilde E^\eps_t)=0,$ the lemmas still hold.

\begin{remark}\label{rem:Cylinder}
We note that Theorem~\ref{thm:CtyPointSoln} generalizes \cite[Theorem~3.1]{G2} to the metric setting and also extends that result to domains in $\R^n$ which are not strictly convex but satisfy the positive mean curvature condition.  For example, consider the domain in $\R^3$ constructed by attaching half of the unit ball to either end of the cylinder $\D\times[0,1].$  The boundary of this capped cylinder has positive mean curvature, but is not strictly convex.  See the discussion from \cite[Section~4]{LMSS} and  \cite[Section~3]{SWZ} relating the notion of positive mean curvature given above to that of domains in $\R^n$ with smooth boundary.   
\end{remark}

\begin{remark}
If we consider a measurable set $F\subset\partial\Omega$ such that $\Ha(\wtil\partial F)=0,$ as in Theorem~\ref{thm:HZeroBoundary}, we see that existence of a solution to the Dirichlet problem with boundary data $\chi_F$ follows immediately from Theorem~\ref{thm:CtyPointSoln}, since $\chi_F$ is continuous at all points $x\in\partial\Omega\setminus\wtil\partial F.$  Thus we obtain another proof of the existence part of Theorem~\ref{thm:HZeroBoundary}, though it is unclear if minimality of the solution also follows from these results, as it does in Section 4. 
\end{remark}

\begin{remark}\label{rem:NotSharp2}
As with Theorem~\ref{thm:SeqApprox} and Theorem~\ref{thm:HZeroBoundary}, we point out that the condition on $f$ of continuity $\Ha$-a.e.\ in Theorem~\ref{thm:CtyPointSoln} is not sharp, as illustrated by Example~\ref{ex:NonLin}.  There, a solution exists for the Dirichlet problem with boundary data $g$, but $g$ is discontinuous on the set $K_{1/4},$ which has positive $\Ha$-measure. 

\end{remark}

\noindent Department of Mathematical Sciences, P.O.~Box 210025, University of Cincinnati, Cincinnati, OH~45221-0025, U.S.A.\\
\noindent E-mail: {\tt klinejp@mail.uc.edu}

\end{document}